\documentclass[hidelinks,onefignum,onetabnum]{siamart220329}

\usepackage{lipsum}
\usepackage{epstopdf}
\ifpdf
  \DeclareGraphicsExtensions{.eps,.pdf,.png,.jpg}
\else
  \DeclareGraphicsExtensions{.eps}
\fi
\usepackage{amsfonts}
\usepackage{mathrsfs}
\usepackage{graphicx}
\usepackage{color}
\graphicspath{{figure/}}
\usepackage{indentfirst,latexsym,bm}
\usepackage{algorithm}
\usepackage{algpseudocode}
\usepackage{amsmath}
\usepackage{extarrows}
\usepackage{hyperref}
\definecolor{orangep}{RGB}{255,128,0}
\definecolor{bluep}{RGB}{0,0,0}
\hypersetup{
    colorlinks=true,
    linkcolor=orangep,
    filecolor=magenta,
    urlcolor=cyan,
    pdftitle={Overleaf Example},
    pdfpagemode=FullScreen,
}
\DeclareMathOperator*{\argmax}{arg\,max}
\DeclareMathOperator*{\argmin}{Arg\,min}
\newcommand{\R}{\mathbb{R}}
\newcommand{\dom}{\mathrm{dom}\,}
\newcommand{\lin}{\mathrm{lin}}
\newcommand{\dist}{\mathrm{dist}}

\newcommand{\aff}{\mathrm{aff}}
\newcommand{\conv}{\mathrm{conv}}
\newcommand{\ri}{\mathrm{ri}}

\newcommand{\bB}{\mathbb{B}}

\newcommand{\cX}{\mathcal{X}}

\newcommand{\cG}{\mathcal{G}}

\newcommand{\cH}{\mathcal{H}}

\newcommand{\vp}{\varphi}

\def\epi{\mathop{\rm epi}}
\def\gph{\mathop{\rm gph}}
\newcommand{\revise}[1]{{\color{bluep}{#1}}}
\makeatletter
\newcommand{\refcheckize}[1]{%
	\expandafter\let\csname
	@@\string#1\endcsname#1%
	\expandafter\DeclareRobustCommand\csname
	relax\string#1\endcsname[1]{%
		\csname
		@@\string#1\endcsname{##1}\wrtusdrf{##1}}%
	\expandafter\let\expandafter#1\csname
	relax\string#1\endcsname
}
\makeatother
\usepackage{enumerate}
\usepackage{cite} 

\usepackage{subfigure}
\usepackage{amssymb}
\usepackage{microtype} 

\newsiamremark{hypothesis}{Hypothesis}
\crefname{hypothesis}{Hypothesis}{Hypotheses}
\newsiamthm{claim}{Claim}

\newsiamthm{thm}{Theorem}
\newsiamthm{prop}{Proposition}
\newsiamthm{coro}{Corollary}
\newsiamthm{lem}{Lemma}
\newsiamthm{assumption}{Assumption}
\newsiamthm{defn}{Definition}

\newsiamremark{rem}{Remark}
\newsiamremark{exmp}{Example}

\headers{KL Exponent via Hadamard parametrization}{W. Ouyang, Y. Liu, T. K. Pong, and H. Wang}

\title{Kurdyka-\L ojasiewicz Exponent via Hadamard parametrization\thanks{Submitted to the editors \today.
\funding{The work of the second author was supported in part by the Natural Science Foundation of Sichuan Province (2022NSFSC1830), and the Southwest Minzu University Research Startup Funds (RQD2022035). The work of the third author was supported by the Hong Kong Research Grants Council PolyU153001/22p. The work of the fourth author was supported by the Natural Science Foundation of Shanghai (21ZR1442800).
}}}

\author{Wenqing Ouyang\thanks{School of Data Science (SDS), Shenzhen Research Institute of Big Data (SRIBD), The
Chinese University of Hong Kong, Shenzhen, People's Republic of China (wenqingouyang1@link.cuhk.edu.cn).}
\and Yuncheng Liu\thanks{School of Mathematics, Southwest Minzu University, Chengdu, Sichuan, People's Republic of China (lyc2015sx@aliyun.com).}
\and Ting kei Pong\thanks{Department  of  Applied  Mathematics, The  Hong  Kong  Polytechnic  University, Hong Kong, People's Republic of China (tk.pong@polyu.edu.hk).}
\and Hao Wang\thanks{School of Information Science and Technology, ShanghaiTech University, Shanghai, People's Republic of China (wanghao1@shanghaitech.edu.cn).}
}

\usepackage{amsopn}
\DeclareMathOperator{\diag}{diag}

\allowdisplaybreaks[2]

\ifpdf
\hypersetup{
  pdftitle={Kurdyka-\L ojasiewicz Exponent via Hadamard parametrization},
  pdfauthor={W. Ouyang, Y. Liu, T. K. Pong, H. Wang}
}
\fi

\begin{document}

\maketitle

\begin{abstract}
  We consider a class of $\ell_1$-regularized optimization problems and the associated smooth ``over-parameterized" optimization problems built upon the Hadamard parametrization, or equivalently, the Hadamard difference parametrization (HDP). We characterize the set of second-order stationary points of the HDP-based model and show that they correspond to some stationary points of the corresponding $\ell_1$-regularized model. More importantly, we show that the Kurdyka-\L ojasiewicz (KL) exponent of the HDP-based model at a second-order stationary point can be inferred from that of the corresponding $\ell_1$-regularized model under suitable assumptions. Our assumptions are general enough to cover a wide variety of loss functions commonly used in $\ell_1$-regularized models, such as the least squares loss function and the logistic loss function. Since the KL exponents of many $\ell_1$-regularized models are explicitly known in the literature, our results allow us to leverage these known exponents to deduce the KL exponents at second-order stationary points of the corresponding HDP-based models, which were previously unknown. Finally, we demonstrate how these explicit KL exponents at second-order stationary points can be applied to deducing the explicit local convergence rate of a standard gradient descent method for \revise{minimizing} the HDP-based model.
\end{abstract}



\section{Introduction}
The idea of introducing redundant or extra number of variables to reformulate an optimization problem has received much attention recently \cite{hoff2017lasso,tibshirani2021equivalences,ziyin2023spred,dai2021representation}. This technique of over-parametrization, though counter-intuitive at first glance, can lead to potential advantages including smoothness of the objective of the over-parameterized model \cite{hoff2017lasso,kolb2023smoothing,poon2023smooth}, better generalization \cite{allen2019learning,pandey2023exploring,subramanian2022generalization} and implicit regularization \cite{tibshirani2021equivalences,zhao2022high,li2021implicit,dai2021representation}.
These benefits have been witnessed and analyzed in various machine learning models including deep neural networks, kernel methods and linear models \cite{belkin2019reconciling, belkin2020two, ergen2021convex,chou2023more}, and there have been tremendous recent interest and progress in developing theory of over-parameterized machine learning that establishes foundational mathematical principles underlying these phenomena; see, e.g., \cite{liu2021we,gunasekar2018characterizing,liu2022loss,sagawa2020investigation}.

In this paper, we focus on a specific type of over-parametrization for a class of $\ell_1$-regularized optimization problems. Specifically, we consider:
\begin{equation}
  \label{l1prob}
  \min_{x\in\R^n} f(x):=h(x) + \mu\|x\|_1,
\end{equation}
where $h\in C^2(\R^n)$ and $\mu > 0$. This class of problems arises naturally in applications such as compressed sensing \cite{candes2005decoding,candes2006stable} and variable selections \cite{tibshirani1996regression}, where $h$ is usually introduced for data fidelity and can be convex or nonconvex depending on the noise model of the data. The associated over-parameterized model we consider is
\begin{equation}
  \label{relaxprob}
  \min_{u,v\in\R^n} G(u,v):=h(u\circ v) + \frac{\mu}{2}(\|u\|^2 + \|v\|^2),
\end{equation}
where $u\circ v$ is the Hadamard (entrywise) product between $u$ and $v$, \revise{and $\|\cdot\|$ denotes the Euclidean norm.} Notice that we have $G(u,v)\ge f(u\circ v)$ for all $(u,v)\in \R^n\times \R^n$ and $\inf f = \inf G$ thanks to the AM-GM inequality. Problem \eqref{relaxprob} was referred to as the Hadamard parametrization of \eqref{l1prob}
in the recent works \cite{hoff2017lasso,poon2023smooth}. \revise{Hadamard parametrization can also be viewed as a formulation of a two-layer diagonal linear network; see the related works \cite{pesme2023saddle,berthier2023incremental,tibshirani2021equivalences}.}


\revise{Interestingly, the function $G$ in \eqref{relaxprob} is smooth}, which has been exploited to design efficient algorithms for solving \eqref{relaxprob}; see, e.g., \cite{hoff2017lasso,poon2023smooth,kolb2023smoothing}. On the other hand, while \eqref{l1prob} is nonsmooth, it is known to have some desirable structural properties for the design of efficient algorithms. Two notable features are:
\begin{enumerate}[{\rm ({$\mathfrak{F}$}I)}]
  \item {\bf (No spurious local minimizers)} When $h$ is convex, so is $f$. Consequently, any optimization methods that return stationary points of $f$ will minimize $f$ when $h$ is convex.
  \item {\bf (Explicitly known KL exponents)} For a large variety of $h$, the function $f$ is known to satisfy the Kurdyka-\L ojasiewicz (KL) property with exponent $\frac12$; see, e.g., \cite{zhou2017unified,li2018calculus}. The KL property is crucial for the convergence analysis of first-order methods \cite{attouch2010proximal,AttBolSva13,BolSabTeb14}. Roughly speaking, a KL exponent of $\frac12$ at stationary points indicates that first-order methods such as the proximal gradient algorithm applied to \eqref{l1prob} is locally linearly convergent; see, e.g., \cite{attouch2009convergence,BolSabTeb14,li2018calculus}.
\end{enumerate}
It is not clear whether features {\rm ({$\mathfrak{F}$}I)} and {\rm ({$\mathfrak{F}$}II)} are inherited by $G$ in \eqref{relaxprob}.
Indeed, for {\rm ({$\mathfrak{F}$}I)}, one can easily observe that $G$ may not be convex even when $h$ is convex. Fortunately, the loss of convexity in $G$ is innocuous when $h$ is a convex quadratic function because, in this case, it can be shown that $G$ satisfies the so-called strict saddle property (see \cite[Appendix~C]{poon2021smooth}), and gradient method with random initialization can provably avoid strict saddle points \cite{lee2019first,poon2023smooth} and converge to second-order stationary points (which are global minimizers of $f$ when $h$ is convex quadratic; see, again, \cite[Appendix~C]{poon2021smooth}) under mild additional assumptions such as coercivity. However, it is not known \emph{whether the strict saddle property holds for $G$ for other convex $h$}. As for {\rm ({$\mathfrak{F}$}II)}, to the best of our knowledge, compared with the many instances of $f$ in \eqref{l1prob} with explicitly known KL exponents, it is not obvious how an explicit KL exponent of $G$ can be obtained except when $h$ is a polynomial. Moreover, it is also unclear \emph{whether one can estimate explicitly the KL exponent of $G$ based on that of $f$} to leverage the many scenarios of $f$ with known KL exponents. Since KL exponents are closely related to convergence rate of first-order methods, intuitively, this also suggests that even when the local convergence rate of standard first-order methods applied to \eqref{l1prob} is known, we do not readily know the local convergence rate of first-order methods applied to \eqref{relaxprob}. In this paper, we study variational properties of $G$ in \eqref{relaxprob} to address the above questions concerning strict saddle property and KL exponents.

To facilitate our discussions, we first describe an equivalent formulation of \eqref{relaxprob} that can be analyzed more readily. Specifically, we apply the following \emph{invertible} linear \revise{transformation} to \eqref{relaxprob}:
\begin{equation}\label{abuv}
a=\frac{u + v}{2}, \ b=\frac{u - v}{2}.
\end{equation}
It is easy to verify that  $a \circ a - b \circ b = u\circ v$. This kind of parametrization is known as Hadamard difference parametrization (HDP), and was studied in \cite{vaskevicius2019implicit,woodworth2020kernel,vivien2022label,kolb2023smoothing}. Using HDP, we can reformulate \eqref{relaxprob} equivalently as
\begin{equation}\label{relaxprob1}
\begin{aligned}
      \min_{a,b\in\R^n} F(a,b)&:=h(a^2-b^2) + \mu\sum_{i=1}^{n}(a_{i}^{2} + b_{i}^{2}) \\
     &=h(a^2-b^2) + \mu\|a^2-b^2\|_1 + 2\mu\|\min\{a^2,b^2\}\|_1,
\end{aligned}
\end{equation}
where the notation $a^2$, $b^2$ and $\min\{a^2,b^2\}$ all denote componentwise operation.
Since \revise{the} invertible linear \revise{transformation} \eqref{abuv} \revise{does} not change the variational properties that we will study in this paper, from now on, we will focus our discussions on $F$ in \eqref{relaxprob1} instead of $G$ in \eqref{relaxprob}.

The specific variational properties of $F$ in \eqref{relaxprob1} that we study in this paper include second-order stationarity and its KL exponent at second-order stationary points. \revise{Our contributions can be summarized as follows:
\begin{itemize}
  \item We show that for every second-order stationary point $(a^*,b^*)$ of $F$, the point $s^* := (a^*)^2 - (b^*)^2$ is a stationary point of $f$ in \eqref{l1prob}. Based on this, we also provide a proof of the strict saddle property of $F$ when $h$ is convex, which extends the result in \cite{poon2021smooth} that studied convex quadratic $h$.
  \item Under a certain strict complementarity condition, we show that if $f$ in \eqref{l1prob} satisfies the KL property with an exponent $\alpha\in (0,1)$ at $s^*$, then $F$ in \eqref{relaxprob1} satisfies the KL property with an exponent $\max\{\alpha,\frac12\}$ at $(a^*,b^*)$.
  \item When the strict complementarity condition fails, by assuming convexity of $h$ and an additional H\"{o}lderian error bound condition with exponent $\gamma\in (0,1]$, we show that if $f$ in \eqref{l1prob} satisfies the KL property with an exponent $\alpha\in (0,1)$ at $s^*$, then $F$ in \eqref{relaxprob1} satisfies the KL property with an exponent $(1+\beta)/2$ at $(a^*,b^*)$, where $\beta = 1 - \gamma(1-\alpha)\in (0,1)$. Examples are provided to demonstrate the tightness of the exponent $\frac{1+\beta}{2}$ when $(\gamma,\alpha)\in \{1\}\times [\frac{1}{2},1)$ or $(\gamma,\alpha)\in
  (0,\frac12]\times (\frac12,1)$.
\end{itemize}}
\noindent For commonly encountered convex $h:\R^n\to \R$ such as the least squares loss function (i.e., $h(x) = \frac12\|Ax - y\|^2$ for some $A\in \R^{m\times n}$ and $y\in \R^m$) and the logistic loss function (i.e., $h(x) = \frac1m\sum_{i=1}^m \log(1 + e^{\langle y_i,x\rangle})$ for some $y_i\in \R^n$, $i = 1,\ldots,m$), the aforementioned H\"{o}lderian error bound condition holds with exponent $\gamma = 1$ and the KL exponent of the corresponding $f$ in \eqref{l1prob} is $\frac12$ (see, e.g., \cite[Corollary~5.1]{li2018calculus}). In these cases, we can deduce that the KL exponent of the corresponding $F$ in \eqref{relaxprob1} at $(a^*,b^*)$ is either $1/2$ or $3/4$, depending on whether the strict complementarity condition holds.

The rest of the paper is organized as follows. In \cref{sec2}, we present notation and discuss some simple properties of $F$ in \eqref{relaxprob1}. Our studies of the second-order stationary points of $F$ and its strict saddle property are presented in \cref{sec31}. Results concerning the KL exponents of $F$ at its second-order stationary points are presented in \cref{subs:uncontrainKLregular}. Finally, we demonstrate in \cref{sec4} how our results in \cref{sec3} can be used to analyze the convergence rate of a standard gradient descent method for solving \eqref{relaxprob1}.

\section{Notation and preliminaries}\label{sec2}
\subsection{\revise{Notation}}
In this paper, we use $\R^n$ to denote the $n$-dimensional Euclidean space, equipped with
the standard inner product $\langle\cdot,\cdot\rangle$ and the induced norm $\|\cdot\|$. We denote the nonnegative orthant by $\R^n_+$. The closed ball centered at $x \in \R^n$ with radius $r$ is denoted by $\mathbb B(x, r)$, and the closed ball centered at the origin with radius $r$ is denoted by $\mathbb B_r$. We say that a set $S\subseteq Q\subseteq \R^n$ has full measure on $Q$ if $Q\setminus S$ has zero Lebesgue measure. If $Q=\R^n$, then we simply say that $S$ has full measure.

For an $x \in \R^n$, we use $\|x\|_1$ to denote its $\ell_1$ norm, and $x^2$ to denote the vector whose $i$th entry is $x_i^2$. Moreover, $\diag(x)$ is the diagonal matrix with $x$ as the diagonal vector. For $x$ and $y\in \R^n$, we use $x\circ y$ and $\min\{x,y\}$ to denote their entrywise product and minimum, respectively. Let $[n]:=\{1,\dots,n\}$. For an $I\subseteq [n]$ and an $x\in \R^n$, $x_I$ is the subvector of $x$ indexed by $I$, and we define $\Pi_I(x)\in \R^n$ as
\begin{equation}
      \label{projindex}
          [\Pi_{I}(x)]_i :=
   \begin{cases}
    x_i  & {\rm if} \ i \in I, \\
    0  & {\rm otherwise}.
   \end{cases}
  \end{equation}
For a matrix $A\in \R^{n\times n}$ and an $I\subseteq [n]$, we use $A_{II}$ to denote the submatrix obtained by extracting the rows and columns of $A$ indexed by $I$. Finally, the $n\times n$ identity matrix is denoted by $I_n$.

For a (nonempty) closed set $D \subseteq \R^n$, its indicator function and support function are respectively defined as
   \[
   \iota_D(x) =
   \begin{cases}
    0  & {\rm if} \ x \in D, \\
    \infty  & {\rm otherwise},
   \end{cases}
   \ \
   {\rm  and}
   \ \
   \sigma_D(x) = \sup\{\langle x, y \rangle : y \in D\}.
   \]
In addition, we denote the distance from an $x \in \R^n$ to $D$ by $\dist(x, D) = \inf_{y \in D} \|x - y\|$, and the convex hull of $D$ is denoted by $\conv(D)$.
For a (nonempty) closed convex set $C\subseteq \R^n$, we write $\ri\,C$ to represent its relative interior. We also use $P_C(x)$ to denote the projection of $x$ onto $C$, which exists and is unique since $C$ is nonempty, closed and convex. The tangent cone $T_C(x)$ for $C$ at an $x \in C$ is defined as
the closure of $ \bigcup_{t > 0} \frac{C - x}{t}$, while the normal cone $N_C(x)$ is defined as $N_C(x) := \{\xi\in \R^n: \langle \xi, y- x\rangle\le 0 \ \ \ \forall y\in C\}$. Recall that $N_C(x)$ is the polar set of $T_C(x)$. Finally, for a closed convex cone $K$, we use $\lin(K):=K\cap -K$ to denote the linearity space of $K$.

For an extended-real-valued function $g : \R^n \to \overline{\R} := \R \cup\{\pm\infty\}$, we denote its epigraph and domain respectively by
\[
   \epi g = \{(x, \alpha)\in \R^n\times \R : g (x) \le \alpha\}
   \ \
   {\rm  and}
   \ \
   \dom g = \{x\in \R^n : g (x) < \infty\}.
   \]
We call $g$ a proper function if $\dom g \neq \emptyset$ and $g (x) > - \infty$ for all $x \in \R^n$. The function $g$ is said to be lower semicontinuous (for short, lsc) if $\epi g$ is closed; we will also say that such a function is closed. \revise{The function $g$ is said to be polyhedral if $\epi g$ is a polyhedral set. We also use the notation $g\in C^k(\R^n)$ to indicate that $g$ has continuous $k$-th order partial derivatives on $\R^n$.} We next recall the definitions of regular subdifferential $\widehat\partial g(x)$ and (limiting) subdifferential $\partial g(x)$ of a proper function $g$ at an $x\in \dom g$:
\begin{align*}
\widehat\partial g(x) &= \{\xi \in \R^n: {\textstyle \liminf_{\tilde x\to x, \tilde x\neq x}\frac{g(\tilde x) - g(x) - \langle\xi,\tilde x - x\rangle}{\|\tilde x - x\|}\ge 0}\},\\
\partial g(x) &= \{\xi \in \R^n:\; \exists x^k \overset{g}\to x \ {\rm and}\ \xi^k \in \widehat \partial g(x^k) \ {\rm with}\ \xi^k \to \xi\},
\end{align*}
where $y \overset{g}\to x$ means both $y \to x$ and $g(y)\to g(x)$; we also set $\partial g(x) = \widehat \partial g(x) = \emptyset$ when $x\notin \dom g$.
For a proper convex function $g$, the above $\partial g$ coincides with the classical notion of convex subdifferential, i.e., for each $x\in \R^n$, we have
\[
\partial g(x) = \{ \xi \in \R^n : g(y) \ge g(x) + \langle \xi, y - x \rangle \ \ \ \forall y \in \R^n\};
\]
see \cite[Proposition~8.12]{rockafellar2009variational}.
Finally, for a set-valued mapping $S : \R^n \rightrightarrows \R^m$, we write
\[
   \gph S = \{(x, u)\in \R^n\times \R^m : u \in S(x)\}
   \ \
   {\rm  and}
   \ \
   \dom S = \{x\in \R^n : S(x)\neq \emptyset\}
   \]
for the graph and the domain of $S$, respectively.

\subsection{\revise{The Kurdyka-\L ojasiewicz property}}
We next recall the definition of Kurdyka-\L ojasiewicz property.
\begin{defn}[Kurdyka-\L ojasiewicz property and exponent] \label{defkl} We say that a proper closed function $g: \R^n \rightarrow \mathbb{R} \cup\{\infty\}$ satisfies the Kurdyka-\L ojasiewicz (KL) property at $\bar{x} \in \dom \partial g$ if there are $a \in(0, \infty]$, a neighborhood $V$ of $\bar{x}$ and a continuous concave function $\varphi:[0, a) \rightarrow[0, \infty)$ with $\varphi(0)=0$ such that
\begin{enumerate}[{\rm (i)}]
  \item $\varphi$ is continuously differentiable on $(0, a)$ with $\varphi^{\prime}>0$ on $(0, a)$;
  \item for any $x \in V$ with $0<g(x)-g(\bar{x})<a$, it holds that
\begin{equation}
    \label{KLineq}
    \varphi^{\prime}(g(x)-g(\bar{x})) \operatorname{dist}(0, \partial g(x)) \geq 1.
\end{equation}
\end{enumerate}
If $g$ satisfies the KL property at $\bar{x} \in \operatorname{dom} \partial g$ and the $\varphi(s)$ in \eqref{KLineq} can be chosen as $\bar{c} s^{1-\alpha}$ for some $\bar{c}>0$ and $\alpha \in[0,1)$, then we say that $g$ satisfies the KL property at $\bar{x}$ with exponent $\alpha$.
\end{defn}

The KL property plays a key role in the convergence analysis of first-order methods (see \cite{AttBolSva13,attouch2010proximal,BolSabTeb14}), and the exponent is closely related to the local convergence rates of these methods (see \cite{attouch2009convergence}). The KL property is known to be satisfied by a large variety of functions \cite{attouch2010proximal}. In particular, it is known that proper closed semialgebraic functions satisfy the KL property with some exponent $\alpha\in [0,1)$; see \cite{bolte2007clarke}.

Our next remark concerns an observation on the neighborhood in Definition~\ref{defkl}(ii), and will be used in our arguments for establishing the KL property later.
\begin{rem}[On the neighborhood requirements for \eqref{KLineq}]
\label{rem2-3}
    Assume $\vp(s)=\bar cs^{1-\alpha}$ for some $\bar c > 0$ and $\alpha \in (0,1)$. Then \eqref{KLineq} can be equivalently written as
    \begin{equation}
        \label{KLexponentineq}
            \bar c(1-\alpha)\,\dist(0,\partial g(x))\geq (g(x)-g(\bar x))^\alpha.
    \end{equation}
    Notice that if there exist a neighborhood $V$ of $\bar x$ and $a > 0$ such that \eqref{KLexponentineq} holds for all $x\in V$ with $\dist(0,\partial g(x))<a$ and $0<g(x)-g(\bar{x})<a$, then, by setting $c=\max\{\bar c,a^{\alpha-1}/(1-\alpha)\}$ and $\vp(s)=cs^{1-\alpha}$, we know that \eqref{KLineq} holds for all $x\in V$ with $0<g(x)-g(\bar x)<a$.\footnote{Indeed, for those $x\in V$ satisfying $\dist(0,\partial g(x))\ge a$ and $0<g(x)-g(\bar{x})<a$, the choice of $c$ readily gives $c(1-\alpha)\,\dist(0,\partial g(x)) \ge c(1-\alpha)\cdot a \ge [a^{\alpha-1}/(1-\alpha)]\cdot (1-\alpha) \cdot a = a^\alpha > (g(x) - g(\bar x))^\alpha$.} Therefore, to verify that $g$ satisfies the KL property at $\bar x$ with exponent $\alpha \in (0,1)$, it suffices to show that there exist $\bar c$, $a > 0$ and a neighborhood $V$ of $\bar x$ such that \eqref{KLexponentineq} holds \revise{for $x\in V$} satisfying $0<g(x)-g(\bar{x})<a$ and $\dist(0,\partial g(x))<a$.
\end{rem}

The next proposition connects the KL property with a local error bound condition and is a direct consequence of ``(iii)$\Rightarrow$(i)" in  \cite[Theorem~2.1]{bai2022equivalence}. We refer the readers to \cite{bolte2017error,drusvyatskiy2018error,garrigos2023convergence} and references therein for related studies.
\begin{prop}[{\cite[Theorem 2.1]{bai2022equivalence}}]\label{prop3-12}
   Let $g:\R^n\to \overline{\R}$ be a proper closed function, and $\bar x$ be a local minimizer of $g$. Let $\Omega:=\{x\in \R^n:~g(x)=g(\bar x)\}$. Suppose $g$ satisfies the KL property at $\bar x$ with exponent $\alpha\in [0,1)$, then there exist a neighborhood $U$ of $\bar x$ and $\sigma>0$ such that $g(x)-g(\bar x)\geq \sigma\dist^{\frac{1}{1-\alpha}}(x,\Omega)$ for all $x\in U$.
\end{prop}


\subsection{Hadamard difference parametrization of $f$}
In this subsection, we collect some useful facts concerning $F$ in \eqref{relaxprob1}. First, we note that the gradient and Hessian of $F$ are given respectively by
\begin{equation}\label{stationary}
\nabla F(a,b)=\left[\begin{matrix}
        2a\circ \nabla h(a^2-b^2)+2\mu a \\
        -2b\circ \nabla h(a^2-b^2)+2\mu b
    \end{matrix}\right],
\end{equation}
\vspace{-0.5 cm}
\begin{equation}
\label{hessianF}
  \begin{aligned}
  \nabla^{2} F(a, b) = &~4\left[
  \begin{matrix}
   \diag(a)  \\
   -\diag(b)
  \end{matrix}
  \right]
  \nabla^2 h(a^2 - b^2)
   \left[
  \begin{matrix}
   \diag(a)  \\
   -\diag(b)
  \end{matrix}
  \right]^{\top}\\
  &+
  2\left[
  \begin{matrix}
   \diag(\nabla h(a^2-b^2)) + \mu I_{n} & \\
    & -\diag(\nabla h(a^2 - b^2)) + \mu I_{n}
  \end{matrix}
  \right].
  \end{aligned}
\end{equation}
This calculation directly yields the next characterization of stationary points of $F$.
\begin{prop}
\label{prop1-2}
Let $F$ and $h$ be defined in \eqref{relaxprob1} and \eqref{l1prob} respectively.
    The point $(a,b)$ is a stationary point of $F$ if and only if for each $i\in [n]$ at least one of the \revise{following} is true:
    \begin{enumerate}[\rm (i)]
        \item $a_i=b_i=0$.
        \item $a_i=0$, $\nabla h(a^2-b^2)_i=\mu$.
        \item $b_i=0$, $\nabla h(a^2-b^2)_i=-\mu$.
    \end{enumerate}
\end{prop}

We note that the equivalence between the local minima of \eqref{l1prob} and \eqref{relaxprob1} in the sense of \cite[Definition~2.2]{kolb2023smoothing} was studied thoroughly in \cite{kolb2023smoothing}.
The next proposition characterizes the local minimizers of \eqref{relaxprob1}, and can be derived from \cite[Theorem 2]{ziyin2023spred} by using the transformation in \eqref{abuv}.
 \begin{prop}[Local minimizers of $f$ and $F$]
 \label{localmin_charac}
   For $(\bar a,\bar b)\in\R^{2n}$, the following statements are equivalent:
    \begin{enumerate}[\rm(i)]
        \item The point $\bar a^2-\bar b^2$ is a local minimizer of $f$ in \eqref{l1prob}, and $\min\{\bar a^2,\bar b^2\} = 0$.
        \item The point $(\bar a,\bar b)$ is a local minimizer of $F$ in \eqref{relaxprob1}.
    \end{enumerate}
\end{prop}


We end this subsection with the following remark, where we describe in item (ii) below an important yet simple observation (a reduction technique) that will allow us to largely simplify our subsequent analysis.
\begin{rem}
\label{remarkab}
\begin{enumerate}[{\rm (i)}]
  \item We deduce from \cref{prop1-2} that $\min\{\bar a^2,\bar b^2\}=0$ at any stationary point $(\bar a,\bar b)$ of $F$ in \eqref{relaxprob1}; see also \cite[Theorem 2]{ziyin2023spred}.
  \item (Reduction technique). \revise{When establishing additional properties of stationary points that are invariant under invertible linear transformations, such as second-order stationarity, local optimality and the KL exponent, we will often introduce a special linear transformation to simplify the argument. We now describe how such a linear transformation is constructed.}

      Let $(\bar a,\bar b)$ be a stationary point of $F$ in \eqref{relaxprob1} and set $\bar s=\bar a^2-\bar b^2$ for notational simplicity.
    Define the following index sets
    \begin{align*}
        &I^1=\{i\in [n]:~\bar a_i>0\}\cup \{i\in [n]:~\bar a_i=\bar b_i=0,\nabla h(\bar s)_i\neq \mu\},\\
        &I^2=\{i\in [n]:~\bar b_i>0\}\cup \{i\in [n]:~\bar a_i=\bar b_i=0,\nabla h(\bar s)_i=\mu\},\\
        &I^3=\{i\in [n]:~\bar a_i<0\},~I^4=\{i\in [n]:~\bar b_i<0\}.
    \end{align*}
    Observe from \cref{prop1-2} (see also \eqref{stationary}) that $\{I^1,I^2,I^3,I^4\}$ forms a partition of $[n]$.
    Next, define the following invertible linear mappings:
    \begin{align*}
        &H(a,b)=(c,d),\ {\rm where}\  (c_i,d_i)=\begin{cases}
          (a_i,b_i) & {\rm if}\ i\in I^1, \\
          (b_i,a_i) & {\rm if}\ i\in I^2, \\
           (-a_i,b_i) & {\rm if}\  i\in I^3, \\
            (-b_i,a_i) & {\rm if}\  i\in I^4,\\
        \end{cases}
        \\&P(s)=t,\ {\rm where}\ t_i=\begin{cases}
            s_i & {\rm if}\ i\in I^1\cup I^3, \\
            -s_i & {\rm if}\ i\in I^2\cup I^4.\\
        \end{cases}
    \end{align*}
    Observe that $P = P^{-1}$. For any $(a,b)\in \R^n\times \R^n$, upon setting $(c,d)=H(a,b)$, $s=a^2-b^2$ and $t=P(s)$, we have from direct computation that $t=c^2-d^2$ and
    \begin{align*}
       &f(s)=f(P\circ P(s))=h\circ P(t)+\mu\|t\|_1 =: \widehat f(t), \\
       & F(a,b)=F(H\circ H(a,b))=h\circ P(c^2-d^2)+\mu \|c\|^2+\mu\|d\|^2 =: \widehat F(c,d),
    \end{align*}
    which means that $\widehat F$ can be viewed as an HDP of $\widehat f$.

    Let $(\bar c,\bar d )=H(\bar a,\bar b)$ and $\bar t =P(\bar s)$. We see that $\bar c\geq 0$, $\bar d=0$, and $\nabla (h\circ P)(\bar t)_i=P(\nabla h(\bar s))_i=-\mu$ whenever $\bar t_i=0$ and $\nabla (h\circ P)(\bar t)_i\in \{-\mu,\mu\}$.\footnote{Indeed, $\bar t_i = 0$ implies $\bar a_i^2 = \bar b_i^2$, which together with item (i) shows that $\bar a_i = \bar b_i = 0$. This implies that $i \in I^1\cup I^2$. Now, if $\nabla (h\circ P)(\bar t)_i\in \{-\mu,\mu\}$, since $\nabla (h\circ P)(\bar t)_i=P(\nabla h(\bar s))_i$, we must have $\nabla h(\bar s)_i\in \{-\mu,\mu\}$ as well. This together with $i\in I^1\cup I^2$ and the definition of $P$ shows that $P(\nabla h(\bar s))_i=-\mu$ as claimed.}
    By replacing $f$ by $\widehat f$ and $F$ by $\widehat F$ if necessary, we could assume $\bar a\geq 0$ and $\bar b=0$ when $(\bar a, \bar b)$ is a stationary point of $F$, and if further $\bar s=\bar a^2-\bar b^2$ is also a stationary point of $f$, we could assume $\nabla h(\bar s)_i=-\mu$ whenever $\bar s_i=0$ and $\nabla h(\bar s)_i\in\{-\mu,\mu\}$. We will state explicitly whether we make such assumptions in our proofs below.
\end{enumerate}
\end{rem}

\section{Variational properties of $F$ from $f$}\label{sec3}

In this section, we analyze some variational properties of $F$ in \eqref{relaxprob1}. \revise{Specifically, we characterize the second-order stationary points of $F$ and relate them to some stationary points of $f$ in \eqref{l1prob}. We also study how the KL exponent of $F$ at those points can be inferred from the KL exponent of $f$.} Since second-order stationarity and KL exponents are invariant under invertible linear transformations, we also obtain a characterization of the second-order stationary points of $G$ in \eqref{relaxprob} and the KL exponent of $G$ at these points.

\subsection{2nd-order stationary points of $F$}\label{sec31}
We first relate the second-order stationary points of $F$ in \eqref{relaxprob1} to some stationary points of $f$ in \eqref{l1prob}.

\begin{prop}[2nd-order stationary points of $F$]\label{prop2ndorder}
Let $f$ and $F$ be defined in \eqref{l1prob} and \eqref{relaxprob1} respectively. Then for all $(a,b)\in\R^n\times \R^n$, the following statements are equivalent:
    \begin{enumerate}[\rm(i)]
        \item The point $s:=a^2 - b^2$ is a stationary point of $f$, $\min\{a^2,b^2\}=0$, \revise{and $\nabla^2h(s)_{II}\succeq 0$, where $I=\{i\in[n]:~s_i\neq 0\}$.}
        \item The point $(a,b)$ is a second-order stationary point of $F$.
    \end{enumerate}
\end{prop}
\begin{proof}
  Before proving the equivalence, \revise{we first note that an $\bar x\in \R^n$ being} a stationary point of $f$ means that $-\nabla h(\bar x)\in \mu\partial \|\bar x\|_1$, which can be expressed as follows:
     \begin{equation}\label{nonsmooth_stationary}
     \begin{cases}
        \nabla h(\bar x)_i = -\mu\cdot\mathrm{sgn}(\bar x_i)  \ \ \ \ \ \ {\rm if}\, \bar x_i \neq 0,\\
        \nabla h(\bar x)_i\in [-\mu, \mu]  \ \  \ \  \ \ \ \  \ \ \  \ \ {\rm if}\, \bar x_i = 0.
     \end{cases}
     \end{equation}

    (i)$\Rightarrow$(ii): We can deduce $\nabla F(a, b) = 0$ from \eqref{stationary}, \eqref{nonsmooth_stationary} and the assumption $\min\{a^2,b^2\}=0$. We next show that $\nabla^{2} F(a, b) \succeq 0$. Since $\nabla F(a, b) = 0$, according to \cref{remarkab}(ii), we may assume without loss of generality that $a\geq 0$ and $b=0$.

    \revise{In this case, we see that} $I=\{i\in [n]: \ a_i>0\}$. \revise{Due to the positive semidefiniteness of $\nabla^2h(s)_{II}$,} we can deduce that (recall that we assumed $b=0$)
    \[
    \begin{aligned}
    &\left[
    \begin{matrix}
     u  \\
     v
    \end{matrix}
    \right]^{\top}\left[
    \begin{matrix}
     \diag(a)  \\
     -\diag(b)
    \end{matrix}
    \right]
    \nabla^2 h(s)
     \left[
    \begin{matrix}
     \diag(a)  \\
     -\diag(b)
    \end{matrix}
    \right]^{\top}\left[
    \begin{matrix}
     u  \\
     v
    \end{matrix}
    \right]
    \\& =  (u\circ a)^\top\nabla^2h(s)(u\circ a)
    = (u_I\circ a_I)^\top\nabla^2h(s)_{II}(u_I\circ a_I) \ge 0\ \ \ \forall (u,v)\in \revise{\R^n\times \R^n}.
    \end{aligned}
    \]
    Furthermore, since $\mu > 0$, we have from \eqref{nonsmooth_stationary} that
    \[
    \pm\diag(\nabla h(s)) + \mu I_{n} \succeq 0.
    \]
  The above two displays together with \eqref{hessianF} show that $\nabla^2 F(a,b)\succeq 0$. This completes the proof of (i)$\Rightarrow$(ii).

     (ii)$\Rightarrow$(i): From \cref{remarkab}(i), we know that $\min\{a^2,b^2\}=0$. In view of \cref{remarkab}(ii), we assume without loss of generality that $b=0$ and $a \ge 0$, and \revise{then} $I:=\{i\in [n]: \ a_i>0\}$. From \eqref{stationary} and $\nabla F(a,b) = 0$, we have
     \begin{equation}\label{gradhmu}
      \nabla h(s)_i = -\mu\ \ \ \forall i \in I.
     \end{equation}
    Thus, to show that $s$ is a stationary point of \eqref{l1prob}, it suffices to prove that $\nabla h(s)_i \in [-\mu,\mu]$ for all $i\notin I$. To this end, we note from $\nabla^{2} F(a, b) \succeq 0$ that
    \[
    [u^{\top}~ v^{\top}]\nabla^{2} F(a, b)\left[
    \begin{matrix}
     u  \\
     v
    \end{matrix}
    \right] \ge 0\ \ \ \forall (u,v)\in\R^n\times \R^n.
    \]
    Recall that $b = 0$. The above display and \eqref{hessianF} imply that for all $(u,v)\in\R^n\times \R^n$,
     \begin{equation}\label{hessian-eq}
       \textstyle 2\sum_{i=1}^n  \left[(\mu + \nabla h(s)_i)u_i^2 + (\mu - \nabla h(s)_i)v_i^2\right] + 4(u\circ a)^\top\nabla^2h(s)(u\circ a)  \ge 0.
     \end{equation}
     By setting $u_I=0$ in \eqref{hessian-eq} so that $u\circ a=0$, we deduce that
     $\nabla h(s)_i\in [-\mu, \mu]$ for all $i \notin I$.
     This together with \eqref{gradhmu} shows that $s$ is a stationary point of $f$.

     Finally, by setting $v=0$ and $u_i=0$ for all $i\notin I$ in \eqref{hessian-eq}, and recalling \eqref{gradhmu} and $a_i>0$ for all $i\in I$, we see further that $\nabla^2h(s)_{II}$ is positive semidefinite.
  \end{proof}

Next, we analyze the \revise{set of minimizers} of convex composite functions. A similar result was established in \cite[Proposition 1]{zhou2017unified} by assuming that $h$ is the sum of a composition of a locally strongly convex and Lipschitz differentiable function with a linear mapping. The proof strategy is inspired by \cite[Theorem 2.1.5]{nesterov2018lectures}. \revise{The proof of \cref{prop2-2} can be found in \cref{appendixA}.}
\begin{lem}\label{prop2-2}
     Assume that $\psi=g+\vp$ with $g,\vp$ being proper closed convex functions, $g\in C^2(\R^n)$, and $\cX:=\argmin\psi\neq\emptyset$. Then, for any $x,y\in \cX$, we have $\nabla g(x)=\nabla g(y)$. Moreover, if we let $v=\nabla g(x)$ for some $x\in \cX$, then $\cX=\nabla g^{-1}(v)\cap \partial \vp^{-1}(-v)$.
\end{lem}

We next present the strict saddle property of $F$ in \eqref{relaxprob1}. \revise{If the strict saddle property holds, then all stationary points are either local minimizers or strict saddle points. Additionally, it has been shown that gradient method with random initialization can provably avoid strict saddle points \cite{lee2019first,poon2023smooth} and converge to second-order stationary points under mild additional assumptions (and hence to local minimizers when the strict saddle property holds).}

When $h$ is convex quadratic, the function $(u,v)\mapsto h(u\circ v)$ was shown to have strict saddle property in \cite[Lemma 1]{zhao2022high}, and $F$ can be shown to have the strict saddle property based on arguments in \cite[Appendix~C]{poon2021smooth}. Our result here covers all convex $h$.

\begin{prop}[Strict saddle property]\label{prop1-3}
Let $f$ and $F$ be defined in \eqref{l1prob} and \eqref{relaxprob1} respectively, and suppose that \revise{$h$} in \eqref{l1prob} is convex. Then \revise{there exists $\delta>0$ } such that for all $(a,b)\in\R^n\times \R^n$, the following statements are equivalent:
    \begin{enumerate}[\rm(i)]
         \item The point $(a,b)$ is a stationary point of $F$ and $\lambda_{\min}(\nabla^2F(a,b))>-\delta$.
        \item The point $a^2 - b^2$ solves \eqref{l1prob}, and $\min\{a^2,b^2\}=0$.
        \item The point $(a,b)$ solves \eqref{relaxprob1}.
        \item The point $(a,b)$ is a second-order stationary point of $F$.
    \end{enumerate}
\end{prop}
\begin{rem}
    The $\delta$ in Proposition~\ref{prop1-3}(i) can be chosen as in \eqref{epsdelta}, which does not depend on the choice of $(a,b)$. \revise{\cref{prop1-3} shows that any stationary point of $F$ in \cref{relaxprob1} is either a global minimizer or a strict saddle point with the smallest eigenvalue of the Hessian being \emph{uniformly} bounded away from $0$; this uniformity was also utilized to prove the complexity of certain perturbed gradient methods in \cite{daneshmand2018escaping,jin2017escape}.}
\end{rem}
\begin{proof}
    For each $I\subseteq [n]$, we define a new function as follows:
    \begin{equation}\label{fI}
    f_I(x) := f(x)+\mu\|x_I\|_1=h(x)+\mu\|x\|_1+\mu\|x_I\|_1.
    \end{equation}
    By applying \cref{prop2-2}, we obtain that for all $x\in\Omega_I:=\argmin_{x\in\R^n}f_I(x)$, the gradient $\nabla h(x)$ remains the same. Specifically, we denote $v^I=\nabla h(x)$ if $\Omega_I$ is nonempty for some $x\in \Omega_I$. Additionally, we set
    \begin{equation}\label{epsdelta}
     \epsilon=\min\{\dist(v^I,[-\mu,\mu]^n):\ I \in {\frak I}\} \ \ {\rm and } \ \  \delta=\min\left\{2\mu,\frac{2\epsilon}{\sqrt{n}}\right\},
    \end{equation}
    where ${\frak I}:= \{I\subseteq [n]: \ \Omega_I\neq\emptyset,~\dist(v^I,[-\mu,\mu]^n)> 0\}$.

    (i)$\Rightarrow$(ii): The condition $\min\{a^2,b^2\} = 0$ follows from \cref{remarkab}(i). Let us define $s=a^2-b^2$. Since $h$ is convex, to show that $s$ solves \eqref{l1prob}, it suffices to show that $-\nabla h(s)\in \mu\partial\|s\|_1$.

    To this end, \revise{we first consider the indices} with $s_i \neq 0$ and proceed with the following two cases:
   \begin{enumerate}[(1)]
        \item $s_i>0$. The conditions $a_i^2 - b_i^2 > 0$ and $\min\{a_i^2,b_i^2\} = 0$ imply that $a_i\neq 0$ and $b_i=0$. From \eqref{stationary} and $\nabla F(a,b) = 0$, we have $2a_i(\nabla h(s)_i+\mu)=0$, which leads to the conclusion that $-\nabla h(s)_i=\mu\in\mu\partial |s_i|$.

         \item $s_i<0$. The proof is similar to case (1).
  \end{enumerate}
  In summary, we have shown that for all $i\in [n]$ with $s_i\neq 0$, we have $-\nabla h(s)_i\in \mu\partial|s_i|$.

  Now it remains \revise{to consider indices $i\in [n]$} with $s_i = 0$. Note that $s_i=0$ is equivalent to $a_i=b_i=0$, as $\min\{a_i^2,b_i^2\}=0$. Next, define
  \begin{equation}\label{tildeI}
  \tilde{I} = \{i\in[n]:\; a_i=b_i=0,\ \nabla h(s)_i \notin [-\mu,\mu]\}.
  \end{equation}
  If $\tilde{I} = \emptyset$, then we can conclude that $s$ solves \eqref{l1prob}. Thus, let's assume to the contrary that $\tilde{I} \neq \emptyset$. We will consider the following two cases:
   \begin{enumerate}[(1)]
        \item Suppose that there exists some $i\in \tilde{I}$ such that $|\nabla h(s)|_i > 2\mu$. We first consider the case where $\nabla h(s)_i< - 2\mu$.
            We define $a(t)_j=a_j$ and $b(t)_j=b_j$ for $j\neq i$, and $a(t)_i=t$ and $b(t)_i=b_i=0$.
          Then, we have $a(t)^2-b(t)^2-(a^2-b^2)=t^2e_{i}$, where $e_i\in \R^n$ is the vector whose $i$th entry is $1$ and is $0$ otherwise, and for all $t>0$, the following holds:
         \begin{align*}
             & F(a(t),b(t))-F(a,b)=h(a(t)^2-b(t)^2)-h(a^2-b^2)+\mu t^2 \\
             &=\langle \nabla h(a^2-b^2), a(t)^2-b(t)^2-a^2+b^2\rangle+O(t^4)+\mu t^2 \\
             &=t^2(\nabla h(a^2-b^2)_i+\mu)+O(t^4) \leq -\mu t^2+O(t^4).
         \end{align*}
         Since $\nabla F(a,b)=0$, this implies  $\frac{1}{2}\begin{bmatrix}e_i^\top &0\end{bmatrix}\nabla^2F(a,b)\begin{bmatrix}
             e_i \\
             0
         \end{bmatrix}\leq -\mu$,  and consequently, we have $\lambda_{\min}(\nabla^2F(a,b))\leq -2\mu\leq -\delta$. This contradicts the assumption that $\lambda_{\min}(\nabla^2F(a,b))>-\delta$.

         The case where $\nabla h(s)_i>2\mu$ can be proved similarly by defining $a(t)_j=a_j$ and $b(t)_j=b_j$ for $j\neq i$, and $a(t)_i=a_i=0$ and $b(t)_i=t$.

        \item For all $i\in \tilde{I}$, it holds that $|\nabla h(s)|_i\leq 2\mu$. In this case, we can verify that $s$ minimizes the $f_{\tilde{I}}$ (i.e., the $f_I $ in \eqref{fI} with $I = \tilde{I}$) by examining the first-order optimality conditions. Furthermore, we know that $\dist(\nabla h(s),[-\mu,\mu]^n)>0$ due to the assumption that $\tilde{I} \neq \emptyset$. Hence, by the definition of $\epsilon$, we obtain that
        \[
        \textstyle\epsilon \le \dist(\nabla h(s),[-\mu,\mu]^n) = \sqrt{\sum_{i \in \tilde{I}}(|\nabla h(s)_i|-\mu)^2} \leq \sqrt{n}\left(\max_{i\in \tilde{I}} |\nabla h(s)_i|-\mu\right).
        \]
       Let us select $\widehat i\in\argmax_{i\in \tilde{I}} |\nabla h(s)_i|$. By employing a similar argument as in case (1),\footnote{Specifically, we define $a(t)_j=a_j$ and $b(t)_j=b_j$ for $j\neq \widehat i$, and set $a(t)_{\widehat i}=t$ and $b(t)_{\widehat i}=0$ if $\nabla h(s)_{\widehat i}  < 0$ and set $a(t)_{\widehat i}=0$ and $b(t)_{\widehat i}=t$ otherwise.} we can show that $\lambda_{\min}(\nabla^2F(a,b))\leq -2(|\nabla h(s)_{\widehat i}|-\mu)\leq -\delta$, which contradicts the assumption that $\lambda_{\min}(\nabla^2F(a,b))>-\delta$.
\end{enumerate}
Therefore, the $\tilde{I}$ in \eqref{tildeI} is empty and we must have that $a^2-b^2$ solves \eqref{l1prob}. This proves (i)$\Rightarrow$(ii). Next, by a similar argument in the proof of \cite[Theorem 2]{ziyin2023spred}, we also have (ii)$\Rightarrow$(iii). Finally, the implication (iii)$\Rightarrow$(iv)$\Rightarrow$(i) is clear.
\end{proof}

\subsection{KL exponent of $F$ at 2nd-order stationary points}\label{subs:uncontrainKLregular}
In this subsection, we study the KL exponent of $F$ in \eqref{relaxprob1} at second-order stationary points. Note that these are global minimizers when $h$ is convex according to \cref{prop1-3}. We will discuss the implications of these exponents on convergence rate in \cref{sec4}.

We start with an auxiliary result that characterizes the KL exponent of a class of functions that can be represented as a sum of a $C^1$ function and a polyhedral function. \revise{The proof of \cref{prop3-10} can be found in \cref{appendixB}.}

\begin{thm}[KL exponent via critical cone]\label{prop3-10}
    Assume that $g\in C^1(\R^n)$ and $\vp:\R^n \to \R\cup\{\infty\}$ is a proper polyhedral function. Let $\bar x$ be a stationary point of $\psi: = g + \vp$ and define
    \begin{align*}
      &K := N_{\partial \vp(\bar x)}(-\nabla g(\bar x)),\ \ \ \tilde \psi(\cdot): = g(\cdot) - \langle \nabla g(\bar x),\cdot - \bar x \rangle + \iota_K(\cdot-\bar x),\\
      &\Omega: = \{x\in\R^n:\ \psi(x) = \psi(\bar x)\},\ \ \  \tilde \Omega: = \{x\in\R^n: \ \tilde \psi(x)=\tilde \psi(\bar x)\}.
    \end{align*}
    Then the following statements hold.
    \begin{enumerate}[\rm (i)]
        \item There exists a neighborhood $U$ of $(\bar x,0)$ such that $U\cap \gph\partial \psi=U\cap \gph\partial\tilde\psi$, and for each $(x,v)\in U\cap \gph\partial \psi=U\cap  \gph\partial\tilde\psi$, it holds that $\psi(x)-\psi(\bar x)=\tilde \psi(x)-\tilde \psi(\bar x)$.
        \item The function $\psi$ satisfies the KL property at $\bar x$ with exponent $\alpha\in (0,1)$ if and only if $\tilde\psi$ satisfies the KL property at $\bar x$ with exponent $\alpha\in (0,1)$.
        \item If $\bar x$ is a local minimizer of $\psi$, then $\bar x$ is also a local minimizer of $\tilde\psi$, and the set $\Omega$ locally coincides with the set $\tilde\Omega$.
    \end{enumerate}
\end{thm}
\begin{rem}[On the set $K$]
    The set $K$ in \cref{prop3-10} is actually the critical cone of $\psi$ at $\bar x$ \cite[Eq.~(3.220)]{bonnans2013perturbation}; indeed, it can be alternatively written as
    \[
    K \overset{\rm (a)}= N_{\partial\psi(\bar x)}(0) \overset{\rm (b)}=\{v\in\R^n:~\sigma_{\partial\psi(\bar x)}(v)=0\}
    \overset{\rm (c)} = \{v\in\R^n:~d\psi(\bar x)(v)=0\},
    \]
    where (a) follows from the fact that $\partial \psi = \nabla g + \partial \vp$ (see \cite[Exercise~8.8(c)]{rockafellar2009variational}; in particular, $\partial \psi(\bar x)$ is a closed convex set), (b) follows from \cite[Example~11.4]{rockafellar2009variational}, (c) follows from \cite[Theorem~8.30]{rockafellar2009variational} and the regularity of $\psi$ deduced from \cite[Corollary~8.19,~Exercise~8.20(b)~and~Proposition~8.21]{rockafellar2009variational}, and \revise{the notation $d\psi(\bar x)(w):=\liminf_{t\downarrow 0,\tilde w\to w}\frac{\psi(\bar x+t\tilde w)-\psi(\bar x)}{t}$ denotes the subderivative of $\psi$ at $\bar x$ for $w$.} Therefore, roughly speaking, according to \cref{prop3-10}, to calculate the KL exponent of $\psi$, we only need to check its behavior on the critical cone.
\end{rem}

Before presenting our results on KL exponents concerning $f$ and $F$ in, respectively, \eqref{l1prob} and \eqref{relaxprob1}, we first derive some useful reformulations of $F$ and $\|\nabla F\|$. Specifically, let $s^*$ be a stationary point of $f$ in \eqref{l1prob}, and define $\tilde h(\cdot)=h(\cdot)-\langle \nabla h(s^*),\cdot - s^*\rangle$. Then, for any $(a,b)\in \R^n\times \R^n$, we can rewrite $F$ in \eqref{relaxprob1} as\vspace{-0.1 cm}
\begin{align}\label{eq2-6}\vspace{-0.1 cm}
               &F(a,b) =h(a^2-b^2)+\mu\sum_{i=1}^n(a_i^2+b_i^2)\notag\\
               &=h(a^2-b^2)-\langle \nabla h(s^*),a^2-b^2\rangle+\!\sum_{i=1}^n\! \left[(\mu+\nabla h(s^*)_i)a_i^2+(\mu-\nabla h(s^*)_i)b_i^2\right] \notag\\
               &=\tilde h(a^2-b^2) - \langle\nabla h(s^*),s^*\rangle +\sum_{i=1}^n \left[(\mu+\nabla h(s^*)_i)a_i^2+(\mu-\nabla h(s^*)_i)b_i^2\right].
\end{align}
Next, let $s=a^2-b^2$ and compute $\nabla F(a,b)$ based on \eqref{eq2-6}, we see that if we define
  \begin{equation}\label{eqgradientF}
  v_i=\begin{bmatrix}
      a_i\nabla \tilde h(s)_i+(\mu+\nabla h(s^*)_i)a_i \\
      -b_i\nabla \tilde h(s)_i+(\mu-\nabla h(s^*)_i)b_i
  \end{bmatrix}\ \ \ \forall i\in [n],
\end{equation}
then the first and second entries of $2v_i$ correspond to the $i$th and $(n+i)$th entries of $\nabla F(a,b)$ (i.e., $\frac{\partial F}{\partial a_i}$ and $\frac{\partial F}{\partial b_i}$), respectively. Also, we have $\|\nabla F(a,b)\|^2 = 4\sum_{i=1}^n \|v_i\|^2$, from which we deduce the existence of $\delta_0>0$ (one can choose $\delta_0 = 2/\sqrt{n}$) such that
\begin{equation}\label{eq2-61}
      \begin{aligned}
     \textstyle
     \|\nabla F(a,b)\| \geq \delta_0\sum_{i=1}^n\|v_i\|.
      \end{aligned}
\end{equation}

\subsubsection{KL exponent under strict complementarity}
The following lemma is an immediate corollary of \cref{prop3-10}. Here, we assume that $-\nabla h(\bar x)\in \ri(\mu\partial \|\bar x\|_1)$, which is typically referred to as the strict complementarity condition for the $f$ in \eqref{l1prob} at the stationary point $\bar x$.
\begin{lem}\label{coro3-11}
    Let $f$ be defined in \eqref{l1prob}, $\alpha \in (0,1)$ and $\bar x$ be a stationary point of $f$ with $-\nabla h(\bar x)\in \ri(\mu\partial \|\bar x\|_1)$. Let $\tilde h(\cdot)=h(\cdot)-\langle \nabla h(\bar x),\cdot - \bar x\rangle$ and $I=\{i\in [n]:~\bar x_i\neq 0\}$. Then $f$ satisfies the KL property at $\bar x$ with exponent $\alpha$ if and only if  there exist $c,r>0$ such that for all $x$ with $\|x-\bar x\|<r$, $0<\tilde h(x)-\tilde h(\bar x)<r$ and $x_{I^c}=\bar x_{I^c}$, it holds that $\|\nabla\tilde h(x)_I\|\geq c(\tilde h(x)-\tilde h(\bar x))^\alpha$.
\end{lem}
\begin{proof}
    Let $\vp(\cdot)=\mu\|\cdot\|_1$. We deduce from $-\nabla h(\bar x)\in \ri(\mu\partial \|\bar x\|_1)=\ri(\partial\vp(\bar x))$ that
    \[
    K:=N_{\partial \vp(\bar x)}(-\nabla h(\bar x))=[\aff(\partial \vp(\bar x)+\nabla h(\bar x))]^\perp=\{v\in \R^n:v_{I^c}=0\}.
    \]
    By \cref{prop3-10}, $f$ satisfies the KL property at $\bar x$ with exponent $\alpha$ if and only if $\tilde h(\cdot) + \iota_K(\cdot -\bar x)$ does so. The rest of the statement is just the definition for the KL property of $\tilde h(\cdot) + \iota_K(\cdot -\bar x)$ \revise{at $\bar x$} with exponent $\alpha$.
\end{proof}

We are now ready to present our main result in this subsection.

\begin{thm}[KL exponent of $F$ under strict complementarity condition]\label{thm2-1}
\revise{Let $(a^*,b^*)$ be a second-order stationary point of $F$ in \eqref{relaxprob1}, and let $s^*=(a^*)^2-(b^*)^2$.} Assume that the function $f$ in \eqref{l1prob} satisfies the KL property at $s^*$ with exponent $\alpha\in (0,1)$.
If $-\nabla h(s^*)\in \ri(\mu\partial\|s^*\|_1)$, then $F$ satisfies the KL property at $(a^*,b^*)$ with exponent $\max\{\alpha,\frac{1}{2}\}$.
\end{thm}

\begin{proof}
  \revise{We first describe the main proof idea. To prove the KL property of $F$ at $(a^*,b^*)$ means we need to connect $\|\nabla F(a,b)\|$ and $F(a,b)-F(a^*,b^*)$ for appropriate $(a,b)$. In view of \eqref{eq2-61}, we calculate the lower bound of each $\|v_i\|$ for $i\in [n]$ to give a lower bound of $\|\nabla F(a,b)\|$ in \eqref{eq2-81}. Then, we provide an upper bound of $F(a,b)-F(a^*,b^*)$ as in \eqref{eqF101}. Finally, we use the KL property of $f$ and \cref{coro3-11} to connect the lower bound of $\|\nabla F(a,b)\|$ and the upper bound of $F(a,b)-F(a^*,b^*)$.
  }

  \revise{We now start our proof.} Since $(a^*,b^*)$ is a second-order stationary point of $F$, by \cref{prop2ndorder} we know that $s^*$ is a stationary point of \eqref{l1prob} with $\min\{(a^*)^2,(b^*)^2\}=0$. Following \cref{remarkab}(ii), we assume without loss of generality that $a^*\geq 0$ and $b^*=0$, and observe that the condition $-\nabla h(s^*)\in \ri(\mu\partial\|s^*\|_1)$ is preserved thanks to the definition of $P$ in \cref{remarkab}(ii).

  Now, we proceed to define the index set $I:=\{i\in [n]:~a^*_i>0\}$. Since $-\nabla h(s^*)\in\ri(\mu\partial \|s^*\|_1)$, \revise{we know that for each} $i\notin I$, it holds that $\nabla h(s^*)_i\in (-\mu,\mu)$. Furthermore, let $\tilde h(\cdot)=h(\cdot)-\langle \nabla h(s^*),\cdot - s^*\rangle$ for notational simplicity.

  From \cref{coro3-11} and the KL assumptions on $f$, we know there exists $c,r>0$ such that for all $x$ with $\|x-s^*\|<r$, $0\le \tilde h(x)-\tilde h (s^*)<r$ and $x_{I^c}= s^*_{I^c}$ it holds that
  \begin{equation}
      \label{kltildeh}
      \|\nabla\tilde h(x)_I\|\geq c(\tilde h(x)-\tilde h(s^*))^\alpha.
  \end{equation}
  In addition, since $a^*_I > 0$, $a^*_{I^c} = 0$, $b^* = 0$ and $\nabla h(s^*)_i \in (-\mu,\mu)$ for all $i\notin I$, \revise{by continuity}, we can find a bounded neighborhood $U$ of $(a^*,b^*)$ to ensure that there exists $\delta > 0$ such that for all $(a,b)\in U$, the following conditions hold:
  \begin{subequations}
  \begin{align}
        &\min_{i\in I}a_i-|b_i| \geq \delta\label{defn_U131}, \\
        &\min_{i\in I^c} \{\mu+\nabla h(s^*)_i,\mu-\nabla h(s^*)_i\}-|\nabla \tilde h(a^2 - b^2)_i|\geq\delta \label{defn_U231},  \\
       &\tilde h([\Pi_{I}(a)]^2)-\tilde h(s^*)<r,~\|[\Pi_{I}(a)]^2-s^*\|<r, \label{defn_U331}\\
         &  \|a_{I^c}\|_\infty\le 1 ,~ \|b\|_{\infty}\leq 1,~\|\nabla F(a,b)\|\leq 1, \label{defn_U431}
  \end{align}
  \end{subequations}
  where $\Pi_I(\cdot)$ is defined in \eqref{projindex}. We also let $L > 0$ denote a Lipschitz continuity modulus for
  $\tilde h$ and $\nabla \tilde h$ on the bounded set $\conv\{a^2-b^2,\ [\Pi_{I}(a)]^2: \ (a,b)\in U\}$.

  Next, we derive a lower bound for $\sum_{i=1}^n\|v_i\|$ (see \eqref{eqgradientF}), which will then be used to give a lower bound for $\|\nabla F(a,b)\|$ on $U$ via \eqref{eq2-61}. To this end, we fix any $(a,b)\in U$ and \revise{let $s = a^2 - b^2$.} We consider two cases.
  \begin{enumerate}[\rm (i):]
  \item $i\in I$. Since $I=\{i\in[n]:\ a_i^*>0\}$ and $b^* = 0$, we have $s^*_i=(a_i^*)^2>0$. Furthermore, since $-\nabla h(s^*)_i \in \mu\partial|s_i^*|$, we have $\nabla h(s^*)_i = -\mu$. We can then rewrite $v_i$ in \eqref{eqgradientF} in the form
  \[
  v_i=
  \begin{bmatrix}
  a_i\nabla\tilde h(s)_i \\
  -b_i\nabla\tilde h(s)_i + 2\mu b_i
  \end{bmatrix},
  \]
  where $s = a^2 - b^2$. Thus,
  \begin{equation*}
  \begin{aligned}
        2\|v_i\|&\geq | a_i\nabla\tilde h(s)_i|+|-b_i\nabla\tilde h(s)_i + 2\mu b_i|\geq a_i|\nabla\tilde h(s)_i|-|b_i||\nabla\tilde h(s)_i|+2\mu|b_i| \\
      &=(a_i-|b_i|)|\nabla\tilde h(s)_i|+2\mu|b_i|  \overset{\rm(a)}{\geq} \delta|\nabla\tilde h(s)_i|+2\mu|b_i|,
  \end{aligned}
  \end{equation*}
  where (a) follows from \eqref{defn_U131}.
  \item
  $i\in I^c$. Notice that $I^c=\{i\in[n]:\ a_i^*=0,\ \nabla h(s^*)_i\in(-\mu,\mu)\}$. We then see from \eqref{defn_U231} that
  \[
    2\|v_i\| \geq |a_i||\nabla \tilde h(s)_i + \mu + \nabla h(s^*)_i|+ |b_i||-\nabla \tilde h(s)_i + (\mu - \nabla h(s^*)_i)| \geq \delta(|a_i|+|b_i|).
  \]
  \end{enumerate}
Summing up the displays from the above two cases for all $i\in [n]$, we have
  \begin{equation}\label{eq2-71}
        2\sum_{i=1}^n\|v_i\|\geq \delta_1\left(\sum_{i\in I}|\nabla\tilde h(s)_i| + \sum_{i\in I^c}|a_i| + \sum_{i=1}^n|b_i|\right),
  \end{equation}
where $\delta_1 =\min\{\delta, 2\mu\} > 0$. Furthermore, combining \eqref{eq2-61} with \eqref{eq2-71}, we see that
  \begin{equation}\label{eq2-81}
    \|\nabla F(a,b)\| \geq \delta_0\sum_{i=1}^n\|v_i\|
    \geq \delta_2\left(\sum_{i\in I}|\nabla\tilde h(s)_i| + \sum_{i\in I^c}|a_i| + \sum_{i=1}^n|b_i|\right),
  \end{equation}
  where $\delta_2 =\frac{1}{2}\delta_0\delta_1$ and $s = a^2 - b^2$.

  Now, we will provide an upper bound for $F(a,b) - F(a^*,b^*)$. Fix any $(a,b)\in U$ with $F(a,b)>F(a^*,b^*)$ and \revise{let $s = a^2 - b^2$.} Let $\tilde s =[\Pi_I(a)]^2$. Then we have
  \begin{equation*}
    s_i - \tilde s_i=\begin{cases}
        -b_i^2     & {\rm if}\ i\in I, \\
        a_i^2 - b_i^2 & {\rm if}\ i\in I^c. \\
    \end{cases}
  \end{equation*}
  This implies that
  \begin{equation}\label{eq2-91}
  \textstyle
   \|s - \tilde s\|\leq \| s - \tilde s\|_1 = \sum_{i\in I}|b_i^2|+\sum_{i \in I^c} |a_i^2 - b_i^2| \leq \sum_{i\in I^c}a_i^2 +\sum_{i = 1}^{n}b_i^2.
  \end{equation}
  Recall the representation of $F$ in \eqref{eq2-6} and recall that $b^* = 0$, we have
  \begin{align}\label{eqF101}
          &F(a,b)-F(a^*,b^*)\notag\\
          &=\tilde h(s) - \tilde h(s^*) + \sum_{i=1}^n(\mu + \nabla h(s^*)_i)\left(a_i^2 - (a_i^*)^2\right) + \sum_{i=1}^n(\mu - \nabla h(s^*)_i)b_i^2 \notag\\
          &\overset{\rm(a)}{=} \tilde h(s)-\tilde h(\tilde s)+\tilde h(\tilde s)-\tilde h(s^*)+\sum_{i\in I^c}(\mu + \nabla h(s^*)_i)a_i^2 + \sum_{i=1}^n(\mu - \nabla h (s^*)_i) b_i^2 \notag\\
          &\overset{\rm(b)}{\leq} L\|s-\tilde s\| +\tilde h(\tilde s)-\tilde h(s^*) + \sum_{i\in I^c}(\mu + \nabla h(s^*)_i)a_i^2+\sum_{i=1}^n(\mu - \nabla h(s^*)_i)b_i^2 \notag\\
          &\overset{\rm(c)}{\leq}\tilde h(\tilde s)-\tilde h(s^*)+\sum_{i\in I^c}(\mu + \nabla h(s^*)_i+L)a_i^2+\sum_{i=1}^n(\mu - \nabla h(s^*)_i+L)b_i^2 \notag\\
          &\overset{\rm(d)}{\leq}\tilde h(\tilde s)-\tilde h(s^*)+\delta_3\left(\sum_{i\in I^c}a_i^2 + \sum_{i=1}^nb_i^2\right)\overset{\rm(e)}{\leq}  \tilde h(\tilde s)-\tilde h(s^*)+\delta_4  \|\nabla F(a,b)\|^2,
  \end{align}
  where (a) holds because when $a_i^* > 0$ (i.e., $i\in I$), we have $\mu + \nabla h(s^*)_i =0$, (b) follows from the Lipschitz continuity of $\tilde h$ with modulus $L$, (c) follows from \eqref{eq2-91}, in (d) we set $\delta_3=\max_{i\in [n]}\{\mu+\nabla h(s^*)_i+L,~\mu-\nabla h(s^*)_i+L\}$, (e) follows from \eqref{eq2-81} and we set $\delta_4=\delta_3/\delta_2^2$. If $\tilde h(\tilde s)-\tilde h(s^*)\leq 0$, \revise{we see immediately from \eqref{eqF101} that}
  \begin{equation}
  \label{fkl1}
       F(a,b)-F(a^*,b^*)
           \leq \delta_4 \|\nabla F(a,b)\|^2.
  \end{equation}
  On the other hand, if $0 < \tilde h(\tilde s)-\tilde h(s^*)$, then we must have $0< \tilde h(\tilde s)-\tilde h(s^*) < r$ and $\|\tilde s - s^*\| < r$ thanks to \eqref{defn_U331}. Since we clearly have $\tilde s_{I^c} = 0 = s^*_{I^c}$, we can invoke \eqref{kltildeh} to show that
\begin{align}\label{eqhnormal}
    &\tilde h(\tilde s)-\tilde h(s^*)\leq c^{-\frac{1}{\alpha}}\| \nabla \tilde h(\tilde s)_I\|^{\frac{1}{\alpha}}\leq c^{-\frac{1}{\alpha}}\left(\|\nabla \tilde h(s)_I\|+\|\nabla \tilde h(s)_I - \nabla \tilde h(\tilde s)_I\|\right)^{\frac{1}{\alpha}}\notag\\
    &\overset{\rm(a)}{\leq} c^{-\frac{1}{\alpha}} \left(\|\nabla \tilde h(s)_I\| + L\|s-\tilde s\| \right)^{\frac{1}{\alpha}} \overset{\rm(b)}{\leq} \delta_5\left(\|\nabla \tilde h(s)_I\|+ \sum_{i\in I^c}a_i^2+\sum_{i=1}^nb_i^2\right)^{\frac{1}{\alpha}} \notag\\
    &\leq \delta_5\left(\|\nabla \tilde h(s)_I\|_1\!+\! \sum_{i\in I^c}a_i^2+\sum_{i=1}^nb_i^2\right)^{\frac{1}{\alpha}}\!\!\!\!=\!\delta_5\left(\sum_{i\in I}|\nabla \tilde h(s)_i|+ \sum_{i\in I^c}a_i^2+\sum_{i=1}^nb_i^2\right)^{\frac{1}{\alpha}} \notag\\
    &\overset{\rm(c)}{\leq} \delta_5\left(\sum_{i\in I}|\nabla \tilde h(s)_i|+ \sum_{i\in I^c}|a_i|+\sum_{i=1}^n|b_i|\right)^{\frac{1}{\alpha}} \overset{\rm(d)}{\leq}\delta_6\|\nabla F(a,b)\|^{\frac{1}{\alpha}},
\end{align}
where (a) follows from the Lipschitz continuity of $\nabla\tilde h$ with modulus $L$, (b) follows from \eqref{eq2-91} with $\delta_5=\left(c^{-1}\max\{L,1\}\right)^{1/\alpha}$, (c) follows from \eqref{defn_U431}, and (d) follows from \eqref{eq2-81} with $\delta_6=\delta_2^{-1/\alpha}\delta_5$. Combining \eqref{eqF101} with \eqref{eqhnormal}, we conclude that
  \[
    F(a,b)-F(a^*,b^*)\leq \delta_4\|\nabla F(a,b)\|^{2}+\delta_6\|\nabla F(a,b)\|^{\frac{1}{\alpha}}\overset{\rm(a)}{\leq } (\delta_4+\delta_6)\|\nabla F(a,b)\|^{\min\{\frac{1}{\alpha},2\}},
  \]
  where in (a) we have used the fact that $\|\nabla F(a,b)\|\leq 1$ (see \eqref{defn_U431}). This together with \eqref{fkl1} proves the KL property of $F$ at $(a^*,b^*)$ with exponent $\max\{\alpha,\frac12\}$.
\end{proof}

\revise{Note that $F$ may not satisfy the KL property at $(a^*,b^*)$ with exponent $\max\{\alpha,\frac{1}{2}\}$ if strict complementarity fails and $\alpha\in [\frac12,1)$, as illustrated by the following example.}

\begin{exmp}
    \label{ce}
    Let $x\in \R$, and $h(x)=(1-\alpha)|x|^{\frac{1}{1-\alpha}}-x$, $\alpha\in [\frac12,1)$ and $\mu=1$ in \eqref{l1prob}. Then $h\in C^2(\R)$ is convex. Clearly, $0$ is a global minimizer of $f(\cdot):=h(\cdot)+|\cdot|$. Therefore, by \cref{prop1-3}, we know $0$ is also a global minimizer of the corresponding $F$ in \eqref{relaxprob1}. By \cref{prop3-10}, we know that the KL exponent of $f$ at $0$ coincides with the KL exponent of $\tilde f:=\tilde h+\iota_K$ at $0$, where $\tilde h(x)=(1-\alpha)|x|^{\frac{1}{1-\alpha}}$ and $K=N_{\revise{\partial |0|}}(-\nabla h(0))=\R_+$. By direct calculation, we can see the KL exponent of $\tilde f$ at $0$ is $\alpha$. On the other hand, we have
    \[  F(a,b)=h(a^2-b^2)+a^2+b^2=(1-\alpha)|a^2-b^2|^{\frac{1}{1-\alpha}}+2b^2.    \]
    Take $t>0$. Then we have
        $\nabla F(t,0)=\begin{bmatrix}
           2t^{\frac{1+\alpha}{1-\alpha}} & 0
        \end{bmatrix}^\top$ and $F(t,0)=(1-\alpha)t^{\frac{2}{1-\alpha}}$.
    This implies that $\|\nabla F(t,0)\|=2(\frac{1}{1-\alpha}F(t,0))^{\frac{1+\alpha}{2}}$, which shows that the KL exponent of $F$ at $0$ is no less than $\frac{1+\alpha}{2}$.
\end{exmp}

\subsubsection{KL exponent in the absence of strict complementarity}
\label{subs:uncontrainklgeneral}
In this subsection, by imposing additional assumptions on $h$ in \eqref{l1prob}, we investigate how the KL exponent of $F$ in \eqref{relaxprob1} at its second-order stationary points can be inferred from the KL exponent of $f$ in \eqref{l1prob} when the strict complementarity condition fails. Our main tool is the following technical lemma. \revise{The proof of \cref{lemma2-5} can be found in \cref{appendixC}.}

\begin{lem}\label{lemma2-5}
   Suppose that $g\in C^2(\R^n)$ is convex and $\bar x \in \R^n$. Define $\psi(\cdot)=g(\cdot) + \iota_K(\cdot-\bar x)$, where $K$ is a closed convex set, and assume that $\bar x $ is a global minimizer of $\psi$. Let $\Omega:=\argmin \psi$. Fix a partition $J_1,J_2,J_3$ of $[n]$. Assume that
   \begin{enumerate}[\rm (i)]
       \item   There exist a neighborhood $U$ of $\bar x$ and constants $c>0$, \revise{$\gamma\in (0,1]$ such that for all $\rho\in \R^{J_3}_{+}$}, it holds that
   \begin{equation}\label{J3eb}
\revise{\dist(x,\Omega\cap S_{\rho}) \leq c\max\{\dist(x,\Omega),\dist(x,S_{\rho})\}^\gamma}\ \ \ \forall x\in U\cap (\bar x + K),
   \end{equation}
   where \revise{$S_{\rho}:=\{x\in\R^n:|x_i-\bar x_i|\leq \rho_i\ \ \ \forall i\in J_3\}$}.
      \item For all $x\in \bar x+K$ and $i\in J_1$, it holds that $x_i=\bar x_i$.
   \end{enumerate}
    Suppose that $\psi$ satisfies the KL property at $\bar x$ with exponent $\alpha\in(0,1)$. 
    Then, by shrinking $U$ if necessary, there is a constant $\sigma>0$ such that
   \[
    \sum_{i\in J_2}|\nabla g(x)_i|^2+\sum_{i\in J_3}|x_i-\bar x_i||\nabla g(x)_i|^2\geq \sigma(g(x)-g(\bar x))^{1+\beta}\ \ \ \forall x\in U\cap (\bar x+K),
   \]
   where \revise{$\beta=1-\gamma(1-\alpha)\in (0,1)$}.
\end{lem}

Before presenting the main result in this subsection, we first introduce some notation. Let $s^*$ be a stationary point of \eqref{l1prob}. We define the following index sets:
  \begin{equation}\label{index}
  \begin{aligned}
            & J_1 = \{i \in [n]: \ s_i^*  =  0, \ \nabla h(s^*)_i \in (-\mu,\mu)\},\\
            & J_2 = \{i \in [n]: \ s_i^*\neq 0, \ \nabla h(s^*)_i \in \{-\mu,\mu\}\},\\
            & J_{3,1} = \{i\in[n]:\ s_i^* = 0,\ \nabla h(s^*)_i = -\mu\},\\
            & J_{3,2}=\{i\in[n]:\ s_i^* = 0,\ \nabla h(s^*)_i = \mu\},\\
            & J_3=J_{3,1}\cup J_{3,2}.
  \end{aligned}
  \end{equation}
  \vspace{-0.5 cm}
\begin{equation}\label{Nindex}
  K := N_{\mu \partial \|s^*\|_1}(- \nabla h(s^*))=\prod_{i\in [n]}K_i, \ \ {\rm where} \ K_i=\begin{cases}
       \{0\}  & i\in J_1, \\
       \R  & i\in J_2, \\
        \R_+ & i\in J_{3,1},\\
        \R_- & i\in J_{3,2}.
    \end{cases}
  \end{equation}
One can check readily that for all $x\in s^* + K$, we have $x_i = s_i^*$ for $i\in J_1$. \revise{Next, we present the result concerning the KL exponent $F$ in \eqref{relaxprob1} based on the KL exponent of $f$ by assuming the convexity of $h$ in \eqref{l1prob} and an additional error bound condition \eqref{nJ3eb}.}

\begin{thm}[KL exponent of $F$ in the absence of strict complementarity]\label{thm3-11}
    Let $(a^*,b^*)$ be a second-order stationary point of $F$ in \eqref{relaxprob1} and assume that $h$ in \eqref{l1prob} is convex. Define $s^*=(a^*)^2-(b^*)^2$ and $\Omega=\{x\in \R^n:\ f(x)=f(s^*)\}$, where $f$ is defined in \eqref{l1prob}. Let $J_1,J_2,J_3$, and $K$ be defined in \eqref{index} and \eqref{Nindex}, respectively. Suppose that there exist a neighborhood $V$ of $s^*$ and constants $c>0$, \revise{$\gamma\in (0,1]$} such that for all \revise{$\rho\in \R_+^{J_3}$}, it holds that
   \begin{equation}\label{nJ3eb}
        \dist(x,\Omega\cap \revise{S_{\rho}}) \leq c\max\{\dist(x,\Omega),\dist(x,\revise{S_{\rho}})\}^\gamma\ \ \ \forall x\in V\cap (s^* + K),
   \end{equation}
   where \revise{$S_{\rho}:=\{x\in\R^n:~|x_i- s_i^*|\leq \rho_i\ \ \forall i\in J_3\}$}. \revise{Suppose also that} $f$ satisfies the KL property at $s^*$ with exponent $\alpha\in (0,1)$. \revise{Then} $F$ satisfies the KL property at $(a^*,b^*)$ with exponent $\frac{1+\beta}{2}$, where \revise{$\beta=1-\gamma(1-\alpha)\in (0,1)$}.
\end{thm}
\begin{rem}[Comments on \eqref{nJ3eb}]
  \label{rem3-12}
    Notice that \revise{$s^* \in \Omega\cap S_\rho$} for all \revise{$\rho\in\R_{+}^{J_3}$}, which implies that \revise{$\Omega \cap S_\rho\neq \emptyset$}.
    By the Hoffman's error bound, condition \eqref{nJ3eb} holds with $\gamma=1$ if $\Omega$ is polyhedral\revise{; see \cite[Lemma 3.2.3]{facchinei2007finite}}. Since $s^*\in \argmin f$ (thanks to \cref{prop1-3} and the assumptions that $h$ is convex and $(a^*,b^*)$ is a second-order stationary point), we know that $\Omega\ (=\argmin f)$ is polyhedral if $h(x)=g(Ax)+\langle c,x\rangle$, where $g:\R^m \to \R$ is a Lipschitz differentiable locally strongly convex function, $c\in \R^n$ and $A\in \R^{m\times n}$; see \cite[Section 4.2]{zhou2017unified}. This covers the commonly used least squares loss function (i.e., $h(x) = \frac12\|Ax - y\|^2$ for some $A\in \R^{m\times n}$ and $y\in \R^m$) and logistic loss function (i.e., $h(x) = \frac1m\sum_{i=1}^m \log(1 + e^{\langle y_i,x\rangle})$ for some $y_i\in \R^n$, $i = 1,\ldots,m$).
\end{rem}
\begin{proof}
\revise{The overall proof strategy is similar to that of \cref{thm2-1}. Instead of using \cref{coro3-11}, here, we use \cref{lemma2-5} to connect the lower bound of $\|\nabla F(a,b)\|$ and the upper bound of $F(a,b)-F(a^*,b^*)$.
}

  \revise{We now start the proof}. Since $(a^*,b^*)$ is a second-order stationary point of $F$ and $h$ is convex, we have from \cref{prop1-3} that $s^*$ is a global minimizer of \eqref{l1prob}, $\min\{(a^*)^2,(b^*)^2\}=0$, and $-\nabla h(s^*) \in \mu\partial \|s^*\|_1$. In view of \cref{remarkab}(ii), we may assume without loss of generality that $a^*\geq 0$, $b^*=0$, and $\nabla h(s^*)_i = -\mu$ whenever $s^*_i = 0$ and $\nabla h(s^*)_i\in \{-\mu,\mu\}$ (hence, $J_{3,2}=\emptyset$). We can then rewrite the index sets in \eqref{index} as follows:
  \begin{equation}\label{simpleJ}
  \begin{aligned}
    & J_1=\{i\in[n]:\ a_i^*=0,\ \nabla h(s^*)_i \in (-\mu,\mu)\},\\
    & J_2=\{i\in[n]:\ a_i^*>0, \ \nabla h(s^*)_i = - \mu\},\\
    & J_{3}=\{i\in[n]:\ a_i^*=0,\ \nabla h(s^*)_i = -\mu \},\\
  \end{aligned}
  \end{equation}
  and observe that \eqref{nJ3eb} holds for the above index sets and the $K$ in \eqref{Nindex} (with $J_{3,2}=\emptyset$) thanks to the definition of $P$ in \cref{remarkab}(ii).

    Next, let $\tilde h(\cdot)=h(\cdot)-\langle \nabla h(s^*),\cdot - s^* \rangle$. According to \cref{prop3-10}, we know that $\tilde f(\cdot): = \tilde h(\cdot) + \iota_K(\cdot-s^*)$ satisfies the KL property at $s^*$ with exponent $\alpha$, and $\Omega$ locally agrees with \revise{$\argmin\tilde f$}. This together with \eqref{nJ3eb} shows that condition (i) in \cref{lemma2-5} holds with $\psi = \tilde f$ and $J_1$, $J_2$ and $J_3$ as in \eqref{simpleJ}. In addition, condition (ii) in \cref{lemma2-5} follows from \eqref{simpleJ} and the definition of $K_i$ in \eqref{Nindex} (with $J_{3,2}=\emptyset$). We can now apply \cref{lemma2-5} with $(\psi,g)=(\tilde f,\tilde h)$ and $J_1$, $J_2$, $J_3$ and $K$ defined in \eqref{simpleJ} and \eqref{Nindex} (with $J_{3,2}=\emptyset$), respectively, to deduce the existence of a constant $\sigma > 0$ and a bounded neighborhood $V_1$ of $s^*$ such that
   \begin{equation}
   \label{kltildehn}
       \sum_{i\in J_2}|\nabla \tilde h(x)_i|^2+\sum_{i\in J_3}|x_i-s_i^*||\nabla \tilde h(x)_i|^2\!\geq\! \sigma(\tilde h(x)-\tilde h(s^*))^{1+\beta}\ \ \ \forall x\in V_1\cap(s^*+K).
   \end{equation}

   Now, in view of \eqref{simpleJ} \revise{and utilizing continuity}, we can take a sufficiently small neighborhood $U$ of $(a^*,b^*)$ to ensure that there exists $\delta > 0$ such that for all $(a,b)\in U$, the following holds:
  \begin{subequations}
  \begin{align}
        &\min_{i\in J_1} \{\mu + \nabla h(s^*)_i,\mu - \nabla h(s^*)_i\} - |\nabla \tilde h(a^2 - b^2)_i|\geq\delta,
        \label{defn_U1new} \\
        & \min_{i\in J_2}a_i - |b_i| \geq \delta\revise{,} \label{defn_U2new}\\
        &\min_{i\in J_3} 2\mu - |\nabla\tilde h(a^2-b^2)_i|\geq \delta, \label{defn_U3new}\\
        &\|a_{J_2^c}\|_\infty\leq 1,~\|b\|_\infty\leq 1, \label{defn_U4new} \\
        &[\Pi_{J_1^c}(a)]^2\in V_1,~ a^2 - b^2 \in V_1,\label{defn_U5new}
  \end{align}
  \end{subequations}
  where $\Pi_{J_1^c}$ is defined as in \eqref{projindex}.

  We now consider three cases, $i\in J_1$, $i\in J_2$, and $i\in J_3$, to provide a lower bound for $\sum_{i=1}^n\|v_i\|$ over $U$, where $v_i$ is given in \eqref{eqgradientF}. To this end, fix any $(a,b)\in U$ and let $s=a^2-b^2$. We then proceed with the three cases.
  \begin{enumerate}[\rm (i):]
  \item $i\in J_1$. We see from \eqref{defn_U1new} and \eqref{eqgradientF} that
  \begin{equation*}
    2\|v_i\| \geq |a_i||\nabla \tilde h(s)_i + \mu + \nabla h(s^*)_i|+ |b_i||-\nabla \tilde h(s)_i + (\mu - \nabla h(s^*)_i)| \geq \delta(|a_i|+|b_i|).
  \end{equation*}

  \item $i\in J_2$. From \eqref{simpleJ}, we see that $s^*_i=(a_i^*)^2 > 0$ and $\nabla h(s^*)_i = -\mu$. We can then rewrite $v_i$ in \eqref{eqgradientF} as
  \[
  v_i=
  \begin{bmatrix}
  a_i\nabla\tilde h(s)_i \\
  -b_i\nabla\tilde h(s)_i + 2\mu b_i
  \end{bmatrix}.
  \]
  Thus,\vspace{-0.2 cm}
  \begin{equation*}
  \begin{aligned}
        2\|v_i\|&\geq | a_i\nabla\tilde h(s)_i|+|-b_i\nabla\tilde h(s)_i + 2\mu b_i| \geq a_i|\nabla\tilde h(s)_i|-|b_i||\nabla\tilde h(s)_i|+2\mu|b_i| \\
      &=(a_i-|b_i|)|\nabla\tilde h(s)_i|+2\mu|b_i| \overset{\rm(a)}{\geq} \delta|\nabla\tilde h(s)_i|+2\mu|b_i|,
  \end{aligned}
  \end{equation*}
  where (a) follows from \eqref{defn_U2new}.

  \item $i\in J_{3}$. We have from \eqref{defn_U3new} and \eqref{eqgradientF} that
  \begin{equation*}
    2\|v_i\| \geq |a_i||\nabla\tilde h(s)_i| + |b_i|(2\mu -|\nabla\tilde h(s)_i|) \geq |a_i||\nabla\tilde h(s)_i| + \delta|b_i|.
  \end{equation*}
  \end{enumerate}
  Summing up the displays from the above three cases for all $i\in [n]$, we have
  \begin{equation}\label{Gradient2}
    2\sum_{i = 1}^n\|v_i\|\geq \delta_1\left(\sum_{i\in J_1}|a_i| + \sum_{i\in J_2}|\nabla\tilde h(s)_i| + \sum_{i\in J_3}|a_i||\nabla\tilde h(s)_i| + \sum_{i = 1}^n|b_i|\right),
  \end{equation}
  where $\delta_1 =\min\{\delta, 1, 2\mu\}$. Now, combining \eqref{eq2-61} with \eqref{Gradient2}, we readily derive a lower bound for $\|\nabla F(a,b)\|$ as follows:
  \begin{align}
    &\textstyle\|\nabla F(a,b)\| \geq \delta_0\sum_{i=1}^n\|v_i\|\notag\\
    &\overset{\rm (a)}\geq \delta_2\left(\sum_{i\in J_1}|a_i| + \sum_{i\in J_2}|\nabla\tilde h(s)_i| + \sum_{i\in J_3}|a_i||\nabla\tilde h(s)_i| + \sum_{i = 1}^n|b_i|\right)\notag\\
     &\ge \delta_2\sqrt{\sum_{i\in J_1}|a_i|^2 + \sum_{i \in J_2}|\nabla\tilde h(s)_i|^2 + \sum_{i\in J_3}|a_i|^2|\nabla\tilde h(s)_i|^2 + \sum_{i=1}^n|b_i|^2}\label{Gradient51}\\
     &\overset{\rm(b)}{\geq} \delta_2\sqrt{\sum_{i\in J_1}|a_i|^{2(1+\beta)} + \sum_{i\in J_2}|\nabla\tilde h(s)_i|^2 + \sum_{i\in J_3}|a_i|^2|\nabla\tilde h(s)_i|^2 + \sum_{i=1}^n|b_i|^{2(1+\beta)}},\label{Gradient52}
  \end{align}
  where (a) holds with $\delta_2 =\frac{1}{2}\delta_0\delta_1$, and (b) follows from \eqref{defn_U4new}.

  Next, we will derive an upper bound for $F(a,b) - F(a^*,b^*)$. Fix any $(a,b)\in U$ with $F(a,b)>F(a^*,b^*)$. Define $s = a^2 - b^2$ and $\tilde s=[\Pi_{J_1^c}(a)]^2$, where $\Pi_{J_1^c}(\cdot)$ is defined as in \eqref{projindex}. Then we have
  \begin{equation}
  \label{sminustildes}
    s_i - \tilde s_i=\begin{cases}
        a_i^2 - b_i^2     & i\in J_1, \\
        - b_i^2 & i\in J_2\cup J_3, \\
    \end{cases}\quad \quad \|s-\tilde s\|^2=\sum_{i\in J_1}(a_i^2-b_i^2)^2+\sum_{i\in J_2\cup J_3}b_i^4.
  \end{equation}
  This implies that
  \begin{equation}\label{boundtildes}
  \textstyle
   \|s - \tilde s\|\leq \| s - \tilde s\|_1 = \sum_{i\in J_1}|a_i^2 - b_i^2|+\sum_{i \in J_2\cup J_3} |- b_i^2| \leq \sum_{i\in J_1}a_i^2+\sum_{i=1}^nb_i^2.
  \end{equation}
  Let $L$ denote the Lipschitz continuity modulus of both $\nabla \tilde h$ and $\tilde h$ on the bounded neighborhood $\conv(V_1)$, where $V_1$ is defined as in \eqref{kltildehn}. Using the representation of $F$ in \eqref{eq2-6} (and recall that $b^* = 0$), we have\vspace{-0.1 cm}
  \begin{align}\vspace{-0.1 cm}
          &F(a,b)-F(a^*,b^*)\notag\\
          &=\tilde h(s) - \tilde h(s^*) + \sum_{i=1}^n(\mu + \nabla h(s^*)_i)\left(a_i^2 - (a_i^*)^2\right) + \sum_{i=1}^n(\mu - \nabla h(s^*)_i)b_i^2 \notag\\
          &\overset{\rm(a)}{=} \tilde h(s)-\tilde h(\tilde s)+\tilde h(\tilde s)-\tilde h(s^*)+\sum_{i=1}^n(\mu + \nabla h(s^*)_i)a_i^2 + \sum_{i=1}^n(\mu - \nabla h(s^*)_i)b_i^2 \notag\\
          &\overset{\rm(b)}{=} \tilde h(s)-\tilde h(\tilde s)+\tilde h(\tilde s)-\tilde h(s^*)+\sum_{i\in J_1}(\mu + \nabla h(s^*)_i)a_i^2 + \sum_{i=1}^n(\mu - \nabla h(s^*)_i)b_i^2 \notag\\
          &\overset{\rm(c)}{\leq} L\|s-\tilde s\|+\tilde h(\tilde s)-\tilde h(s^*)+\sum_{i\in J_1}(\mu + \nabla h(s^*)_i)a_i^2+\sum_{i=1}^n(\mu - \nabla h(s^*)_i)b_i^2 \notag\\
          &\overset{\rm(d)}{\leq} \tilde h(\tilde s)-\tilde h(s^*)+\sum_{i\in J_1}(\mu + \nabla h(s^*)_i + L)a_i^2+\sum_{i=1}^n(\mu - \nabla h(s^*)_i + L)b_i^2 \label{eqF1sub}\\
          &\overset{\rm(e)}{\leq} \tilde h(\tilde s)-\tilde h(s^*)+\delta_3\|\nabla F(a,b)\|^2,\label{eqF1}
  \end{align}
  where (a) holds because $a_i^* > 0$ implies $\mu + \nabla h(s^*)_i =0$ (see \eqref{simpleJ}), (b) holds because $\mu + \nabla h(s^*)_i =0$ except when $i \in J_1$, (c) follows from the Lipschitz continuity of $\tilde h$ on $\conv(V_1)$ with modulus $L$ and \eqref{defn_U5new} (so that $s$ and $\tilde s\in V_1$), (d) follows from \eqref{boundtildes}, (e) follows from \eqref{Gradient51} and we set
  $\delta_3=\delta_2^{-2}\cdot\max_{i\in [n]}\{\mu+\nabla h(s^*)_i+L,\mu-\nabla h(s^*)_i+L\}$.

  Now, if $\tilde h(\tilde s)-\tilde h(s^*)\leq 0$, we see from \eqref{eqF1} that\vspace{-0.1 cm}
  \begin{equation}\label{kleq0}\vspace{-0.1 cm}
      F(a,b)-F(a^*,b^*)\leq  \delta_3\|\nabla F(a,b)\|^2.
  \end{equation}
  Otherwise, using the Lipschitz continuity of $\nabla \tilde h$ on $\conv(V_1)$, \eqref{kltildehn}, \eqref{defn_U5new} (so that $s$ and $\tilde s\in V_1$) and noting that $\tilde s \in s^* + K$, we have
  \begin{align}\label{eqF2}
    &\sigma\left(\tilde h(\tilde s) - \tilde h(s^*)\right)^{1 + \beta}\leq  \sum_{i\in J_2}|\nabla \tilde h(\tilde s)_i|^2 + \sum_{i\in J_3}|\tilde s_i-s^*_i||\nabla \tilde h(\tilde s)_i|^2 \notag\\
    &\!\!\overset{\rm(a)}{=}\sum_{i\in J_2}|\nabla \tilde h(\tilde s)_i - \nabla \tilde h(s)_i + \nabla \tilde h(s)_i|^2 + \sum_{i\in J_3}a_i^2|\nabla \tilde h(\tilde s)_i - \nabla \tilde h(s)_i + \nabla \tilde h(s)_i|^2 \notag\\
    &\!\!\overset{\rm(b)}{\leq}\! 2\!\!\sum_{i\in J_2}\!(|\nabla \tilde h(s)_i|^2 \!\!+\! |\nabla \tilde h(s)_i \!-\! \nabla \tilde h(\tilde s)_i|^2) \!+\! 2\!\!\sum_{i\in J_3}\!\!a_i^2(|\nabla \tilde h(s)_i|^2 \!\!+\! |\nabla \tilde h(s)_i\!-\!\nabla \tilde h(\tilde s)_i|^2) \notag\\
    &\!\!\overset{\rm(c)}{\leq} \!2\!\!\sum_{i\in J_2}\!|\nabla \tilde h(s)_i|^2 \!\!+\! 2\sum_{i\in J_3}\!a_i^2|\nabla \tilde h(s)_i|^2 \!+\! 2L^2\!\!\sum_{i\in J_2}\!\|s_i- \tilde s_i\|^2 \!\!+\! 2L^2\!\sum_{i\in J_3}\!a_i^2\|s_i- \tilde s_i\|^2\notag\\
    &\!\!\overset{\rm(d)}{\leq} 2\sum_{i\in J_2}|\nabla \tilde h(s)_i|^2 + 2\sum_{i\in J_3}a_i^2|\nabla \tilde h(s)_i |^2 + 2L^2\|s- \tilde s\|^2\notag\\
    &\!\!\overset{\rm(e)}{\leq} 2\sum_{i\in J_2}|\nabla \tilde h(s)_i|^2 + 2\sum_{i\in J_3}a_i^2|\nabla \tilde h(s)_i |^2 + 4L^2\left(\sum_{i\in J_1}a_i^4 + \sum_{i=1}^n b_i^4 \right),
  \end{align}
 where (a) follows from $\tilde s_i - s^*_i = a_i^2$ for $i \in J_3$, (b) holds because of the fact that $(x+y)^2\leq 2x^2+2y^2$ for all $x,y\in \R$, (c) follows from the Lipschitz continuity of $\nabla \tilde h$, (d) follows from \eqref{defn_U4new} (so that $|a_i| \le 1$ for $i\in J_2^c$), and (e) follows from the second relation in \eqref{sminustildes} and we have used the inequality $(x+y)^2\leq 2x^2+2y^2$ again. Consequently, using the convexity of $|\cdot|^{1+\beta}$ and \eqref{eqF1sub}, we deduce that there is a $\delta_4>0$ which is independent of $(a,b)\in U$ such that\vspace{-0.2 cm}
 \begin{align*}\vspace{-0.2 cm}
   &\left( F(a,b) - F(a^*,b^*)\right)^{1+\beta} \leq \delta_4 \left(\left(\tilde h(\tilde s) - \tilde h(s^*)\right)^{1+\beta}+\sum_{i\in J_1}|a_i|^{2(1+\beta)}+\sum_{i=1}^n |b_i|^{2(1+\beta)}\right) \\
   &\overset{\rm(a)}{\leq} \delta_5\!\left(\sum_{i\in J_2}\!|\nabla \tilde h(s)_i|^2 \!+\!\! \sum_{i\in J_3}a_i^2|\nabla \tilde h(s)_i|^2 \!+\!\! \sum_{i\in J_1}a_i^4+\sum_{i=1}^nb_i^4 \!+\!\! \sum_{i\in J_1}|a_i|^{2(1+\beta)} \!+\!\! \sum_{i=1}^n |b_i|^{2(1+\beta)}\!\right) \\
   &\overset{\rm(b)}{\leq} 2\delta_5\!\left(\sum_{i\in J_1}\!|a_i|^{2(1+\beta)}\!\!+\!\!\sum_{i\in J_2}\!|\nabla \tilde h(s)_i|^2\!\!+\!\!\sum_{i\in J_3}\!\!a_i^2|\nabla \tilde h(s)_i|^2\!\!+\!\!\sum_{i=1}^n |b_i|^{2(1+\beta)}\!\right) \!\overset{\rm(c)}{\leq}\! \delta_6\|\nabla F(a,b)\|^2,
 \end{align*}
 where (a) follows from \eqref{eqF2} and we set $\delta_5=\delta_4\max\{1,\sigma^{-1}\max\{2,4L^2\}\}$, (b) follows from \eqref{defn_U4new}, (c) follows from \eqref{Gradient52} by setting $\delta_6=2\delta_2^{-2}\delta_5$. This together with \eqref{kleq0} implies that $F$ satisfies the KL property at $(a^*,b^*)$ with exponent $\frac{1+\beta}{2}$.
\end{proof}
\begin{rem}[Tightness of the exponent $\frac{1+\beta}2$ in \cref{thm3-11}]
    When $\gamma=1$, the constant $\beta$ given in \cref{thm3-11} becomes $\alpha$. Therefore, under the settings of \cref{thm3-11}, the function $F$ satisfies the KL property at the second-order stationary point with exponent $\frac{1+\alpha}{2}$, which happens to be the exponent given in \cref{ce}. \revise{Moreover, the set of minimizers in \cref{ce} is a singleton, so the error bound condition \eqref{nJ3eb} holds with $\gamma=1$ by \cref{rem3-12}.} Consequently, the exponent given in \cref{thm3-11} is tight when $\gamma=1$ and $\alpha \in [\frac12,1)$. \revise{We would also present the next example to show that this exponent is tight when $\alpha\in (\frac12,1)$ and $\gamma\in (0,\frac{1}{2}]$.}
\end{rem}
\revise{
 \begin{exmp}
  \label{cen}
  Let $x\in \R^2$, and $h(x)=(1-\alpha)(|x_1|^\frac{1}{\gamma}-x_2)_+^{\frac{1}{1-\alpha}}-x_1-x_2$ for $x=(x_1,x_2)\in\R^2$ with $d_+:=\max\{d,0\}$ for $d\in \R$, $0<\gamma\leq\frac{1}{2}$, $\alpha\in (\frac{1}{2},1)$ and $\mu=1$. We would like to show that the KL exponent of the corresponding $f$ in \eqref{l1prob} is $\alpha$ at $0$. Clearly, $h\in C^2(\R^2)$ and is convex. Moreover, $0$ is a global minimizer of $f$. Therefore, we know $0$ is also a global minimizer of $F$ in \eqref{relaxprob1} by \cref{prop1-3}. By \cref{prop3-10}, we know the KL exponent of $f$ at $0$ coincides with the KL exponent of $\tilde f:=\tilde h+\iota_K$ at $0$, where $\tilde h(x)=(1-\alpha)(|x_1|^\frac{1}{\gamma}-x_2)_+^{\frac{1}{1-\alpha}}$ and $K=N_{\partial \|0\|_1}(-\nabla h(0))=\R^2_+$. Take $x$ to be sufficiently close to $0$ with $x\in K$, then $x_1,x_2\geq 0$ and we have:
  \begin{align*}
    \dist(0,\partial\tilde f(x))&=\dist(-\nabla \tilde h(x),N_K(x)) \geq \dist((x_1^{\frac{1}{\gamma}}-x_2)_+^{\frac{\alpha}{1-\alpha}},N_{\R_+}(x_2))\\
    &\overset{\rm (a)}{=}(x_1^{\frac{1}{\gamma}}-x_2)_+^{\frac{\alpha}{1-\alpha}}=(\tilde h(x)-\tilde h(0))^{\alpha}/(1-\alpha)^\alpha,
  \end{align*}
   where in (a) we have used $0\in N_{\R_+}(x_2)\subseteq \R_-$. Therefore, the KL exponent of $f$ at $0$ is $\alpha$. This exponent is tight since for all $x\in K$ with $x_1\leq 1$ we have:
   \begin{equation}
    \label{bound_tildef}
    \begin{aligned}
      &\dist(0,\partial\tilde f(x))\leq \|\nabla \tilde h(x)\|=\left\|\begin{bmatrix} \frac{1}{\gamma}x_1^{\frac{1}{\gamma}-1} (x_1^\frac{1}{\gamma}-x_2)_+^{\frac{\alpha}{1-\alpha}} \\   -(x_1^\frac{1}{\gamma}-x_2)_+^{\frac{\alpha}{1-\alpha}}   \end{bmatrix}   \right\| \\
      &\leq\sqrt{1/\gamma^2+1}(\tilde h(x)-\tilde h(0))^{\alpha}/(1-\alpha)^\alpha.
    \end{aligned}
   \end{equation}

   Next, let $a=(a_1,a_2),b=(b_1,b_2)\in \R^2,~Q(a,b):=(|a_1^2-b_1^2|^{\frac1\gamma}-a_2^2+b_2^2)_+^{\frac{\alpha}{1-\alpha}}$, and $T(a,b):=\frac{2}{\gamma}\mathrm{sgn}(a_1^2-b_1^2)|a_1^2-b_1^2|^{\frac1\gamma-1}$, where
   \[
   \mathrm{sgn}(y)=\begin{cases} \frac{y}{|y|} & ~\text{if }y\neq 0, \\ ~0 & \text{~if }y=0. \end{cases}
    \]
    Then we know $F$ in \eqref{relaxprob} and $\nabla F$ can be written as
  \begin{align*}
      F(a,b)&=h(a^2-b^2)+\|a\|^2+\|b\|^2\\
      &=(1-\alpha)(|a_1^2-b_1^2|^{\frac1\gamma}-a_2^2+b_2^2)_+^{\frac{1}{1-\alpha}}+2b_1^2+2b_2^2, \\
      \nabla F(a,b)&=\begin{bmatrix}
         a_1T(a,b)Q(a,b) \\ -2a_2Q(a,b)\\ 4b_1-b_1T(a,b)Q(a,b)\\ 4b_2+2b_2Q(a,b)
      \end{bmatrix}.
  \end{align*}
 Let us evaluate the gradient $\nabla F((t,0),0)$ and function value $F((t,0),0)$ with $t>0$:
  \begin{align*}
      \nabla F((t,0),0)=\begin{bmatrix}
          \frac{2}{\gamma}t^{\frac{2}{\gamma(1-\alpha)}-1}   & 0 &
          0 & 0
      \end{bmatrix}^\top ,\quad F((t,0),0)=(1-\alpha)t^{\frac{2}{\gamma(1-\alpha)}}.
  \end{align*}
  Clearly, we have $\|\nabla F((t,0),0)\|=\frac{2}{\gamma}(F((t,0),0)-F(0,0))^{\frac{2-\gamma(1-\alpha)}{2}} /(1-\alpha)^{\frac{2-\gamma(1-\alpha)}{2}}$, which shows that the KL exponent of $F$ at $0$ is no less than $\frac{1+\beta}{2}$ with $\beta=1-\gamma(1-\alpha)$.

  We next verify the error bound condition \cref{nJ3eb}. Notice that $\Omega:=\argmin f=\{x\in \R^2_+:~x_2\geq x_1^{\frac1\gamma}\}=\argmin\tilde f$. Taking arbitrary $x\in K$ with $x_1\leq 1$ and defining $y:=P_{\Omega}(x)$, we have
  \begin{align*}
      &\tilde f(x)-\tilde f(0)=\tilde f(x)-\tilde f(y)\overset{\rm (a)}{\leq}\langle \nabla \tilde h(x),x-y\rangle\leq \|\nabla \tilde h(x)\|\dist(x,\Omega) \\
      &\overset{\rm (b)}{\leq}\sqrt{1/\gamma^2+1}(\tilde f(x)-\tilde f(0))^{\alpha}\dist(x,\Omega)/(1-\alpha)^\alpha,
  \end{align*}
  where in (a) we have used the convexity of $\tilde f$ and $\nabla \tilde h(x)\in\partial \tilde f(x)$, and in (b) we have used the second inequality in \cref{bound_tildef} and the fact $\tilde f(x)-\tilde f(0)=\tilde h(x)-\tilde h(0)$ due to $x\in K$. Rearranging terms in the above inequality, we see that for all $x\in K$ with $x_1\leq 1$ it holds that:
  \begin{align}
    \label{bound_dist}
       \sigma\dist(x,\Omega)\geq  (x_1^{\frac1\gamma}-x_2)_+,~ \text{where } \sigma=\sqrt{1/\gamma^2+1}/(1-\alpha).
  \end{align}
   Now, for an arbitrary $\rho\in \R^2_+$, we define the set $\Omega_{\rho}:=\Omega\cap S_\rho$, where $S_\rho:=\{x\in \R^2:~ |x_i|\leq \rho_i~~\forall i=1,2\}$. For any $x\in K$ with $\max\{x_1,x_2\}\leq 1/2$, define $z:=P_{S_\rho}(x)$. Then it can be verified that $z\in \R_+^2$ and $\max\{z_1,z_2\}\leq \max\{x_1,x_2\}\leq 1/2$ since $x\in \R_+^2$. We further define $p\in \Omega_\rho$ in the following manner and calculate $z-p$:
   \begin{equation}\label{defp}
   p=(\min\{z_2^\gamma,z_1\},z_2),    \quad z-p=   \\
     (\max\{z_1-z_2^\gamma,0\},0).
     \end{equation}
  Then we have:
\begin{align}
  \label{boundzp}
  \|z-p\|=\max\{z_1-z_2^\gamma,0\}\overset{\rm(a)}{\leq} (z_1^\frac1\gamma-z_2)_+^\gamma\overset{\rm (b)}{\leq}\sigma^\gamma \dist^\gamma(z,\Omega),
\end{align}
  where (a) clearly holds if $z_1\le z_2^\gamma$, while when $z_1>z_2^\gamma$, we deduce (a) from $z_1-z_2^\gamma=(z_1^\frac1\gamma)^\gamma-z_2^\gamma\leq (z_1^\frac1\gamma-z_2)^\gamma$,\footnote{\label{footnote4}\revise{The second inequality follows from the fact that $(c+d)^\gamma \le c^\gamma + d^\gamma$ for all $c$, $d\ge 0$; this last inequality clearly holds if $c=0$ or $d=0$, and when $c>0$ and $d>0$ it holds because $1 = \frac{c}{c+d} + \frac{d}{c+d}\le (\frac{c}{c+d})^\gamma + (\frac{d}{c+d})^\gamma$, thanks to $\gamma \in (0,\frac12]$.}} and (b) follows from \eqref{bound_dist} thanks to $z\in K$ and $z_1\leq 1$.
Therefore, we have any $x\in K$ with $\max\{x_1,x_2\}\leq 1/2$ and $z:=P_{S_\rho}(x)$ that
  \begin{align*}
     &\dist(x,\Omega_\rho)\leq \|x-z\|+\dist(z,\Omega_\rho)\overset{\rm (a)}{\leq} \dist(x,S_\rho)+\|z-p\| \\
     &\overset{\rm (b)}{\leq} \dist(x,S_\rho)+ \sigma^\gamma\dist^\gamma(z,\Omega)\leq \dist(x,S_\rho)+ \sigma^\gamma(\dist(x,\Omega)+\|x-z\|)^\gamma \\
     &\overset{\rm (c)}{\leq} \dist(x,S_\rho)+\sigma^\gamma(\dist^\gamma(x,\Omega)+\dist^\gamma(x,S_\rho))\overset{\rm (d)}{\leq} c \max\{\dist(x,S_\rho),\dist(x,\Omega)\}^\gamma,
  \end{align*}
  where (a) holds in view of $z=P_{S_\rho}(x)$ and the $p\in \Omega_\rho$ constructed in \eqref{defp}, (b) holds because of \eqref{boundzp}, in (c) we have used $z=P_{S_\rho}(x)$ and the inequality $(c+d)^\gamma \le c^\gamma + d^\gamma$ for any $c$, $d\ge 0$ (see footnote~\ref{footnote4}), and in (d) we have used $\max\{x_1,x_2\}\leq 1/2$ and $0\in S_\rho$ to show that $\dist(x,S_\rho)\leq \|x\|<1$ and set $c=1+2\sigma^\gamma$. Consequently, we know the condition \eqref{nJ3eb} holds, and by \cref{thm3-11} we know the KL exponent of $F$ is $\frac{1+\beta}{2}$ with $\beta=1-\gamma(1-\alpha)$, which is tight by the previous argument.
\end{exmp}
}

\begin{rem}[Explicit KL exponent]
   When $h$ is the least squares loss function or the logistic loss function as described in \cref{rem3-12}, the KL exponent of the corresponding $f$ in \eqref{l1prob} is $\frac12$ (see \cite[Corollary~5.1]{li2018calculus}). In these cases, by \cref{thm2-1}, \cref{thm3-11} and \cref{rem3-12}, we can deduce that the KL exponent of the corresponding $F$ in \eqref{relaxprob1} at a second-order stationary point $(a^*,b^*)$ is either $1/2$ or $3/4$, depending on whether $-\nabla h(s^*)\in \ri(\mu\partial\|s^*\|_1)$, where $s^* = (a^*)^2 - (b^*)^2$.
\end{rem}

\section{Convergence analysis of a first-order method}\label{sec4}

\begin{algorithm}
    \begin{algorithmic}
\caption{ Gradient descent with backtracking linesearch}\label{algo1}
\Require{Initial $(a^0,b^0)\in\R^n\times \R^n$, initial stepsize $\theta_0\in (0,\infty)$, $\kappa\in (0,1)$.} \\
$\theta=\theta_0$\;
\For{ $k=0,1,\dots$ }
\State $(a^{k+1},b^{k+1})=(a^k,b^k)-\theta\nabla F(a^k,b^k)$\;
\While{$F(a^{k+1},b^{k+1})>F(a^k,b^k)-\frac{\theta^2}{2}\|\nabla F(a^k,b^k)\|^2$}
\State $\theta\gets \kappa\theta$\;
\State $(a^{k+1},b^{k+1})=(a^k,b^k)-\theta\nabla F(a^k,b^k)$\;
\EndWhile
\EndFor
\end{algorithmic}
\end{algorithm}

\revise{To highlight the consequences of our results in \cref{sec3}}, we analyze the convergence rate of a standard first-order method applied to minimizing $F$ in \eqref{relaxprob1}, where we incorporate a standard linesearch scheme to adaptively estimate the ``local" Lipschitz constant of $\nabla F$; see \cref{algo1}. We will argue that, under some reasonable assumptions, \revise{with} almost every initial point \revise{and initial stepsize}, the sequence generated by \cref{algo1} converges to a second-order stationary point and the convergence rate can then be deduced based on the KL exponent of $f$ in \eqref{l1prob}, according to \cref{thm2-1} and \cref{thm3-11}. The reduction to the KL exponents of $f$ is important because the explicit KL exponents of functions of the form $f$ have been identified for a large variety of convex $h$ that arise in contemporary applications (see, e.g., \cite{zhou2017unified,bolte2017error,drusvyatskiy2018error,li2018calculus}), while not much is known concerning the explicit \revise{KL exponents of the corresponding $F$.}

We first present our assumption. (\rm{A1}) is inspired by \cite[Assumption 3]{cheridito2022gradient}.
\begin{assumption}
\label{assum3-1}
Let $F$ be defined in \eqref{relaxprob1} and consider \cref{algo1}. Let $\Gamma=\{\kappa^k\theta_0:~k\in \mathbb N\}$, and define $\cG_\theta(a,b):=(a,b)-\theta\nabla F(a,b)$. We assume that
\begin{enumerate}[\rm ({A}1)]
    \item For each $\theta\in \Gamma$, there exists an open set $U_\theta\subseteq \R^{2n}$, whose complement has zero Lebesgue measure, and $\det(D\cG_\theta)$ is nonzero on $U_\theta$.\footnote{Here, $D\cG_\theta$ is the Jacobian of $\cG_\theta$.}
    \item The function $F$ is level bounded.
    \item The function $F$ satisfies the KL property at each of its stationary point.
\end{enumerate}
\end{assumption}

Under the above assumption, we will show that \cref{algo1} can converge to a second-order stationary point of $F$ in some almost sure sense. Let us first justify these assumptions. We note that (A2) holds if $h$ is lower bounded, in which case $F$ is the sum of a level bounded function and a lower bounded function, and hence is level bounded. (A3) is a standard assumption in the convergence analysis of first-order method; see, e.g., \cite{attouch2010proximal,AttBolSva13,BolSabTeb14}. For (A1), we will show it holds if $F$ is \revise{subanalytic}: this covers the cases when $h$ is the least squares loss function or the logistic loss function as described in \cref{rem3-12}, \revise{and most other practical loss functions. The definition of subanalytic sets and functions are taken from \cite[p. 40]{shiota1997geometry}, and is equivalent to the projection definition of semianalytic sets due to \cite{hironaka1973subanalytic}; see, e.g., \cite[Proposition 3.13]{bierstone1988semianalytic}.
\begin{definition}
  \label{defn_subanalytic}
  A set $X\subseteq \R^n$ is said to be subanalytic, if for each $x\in \R^n$, there exists a neighborhood $U$ of $x$ such that the set $U\cap X$ can be written as finite union of sets in the form $\mathrm{Im}(f_1)-\mathrm{Im}(f_2)$,\footnote{\revise{Here, ${\rm Im}(f)$ is the image of the function $f$.}} where $f_1$ and $f_2$ are topologically proper (i.e. the preimage of compact set is compact), and real analytic from real analytic manifolds to $\R^n$. A mapping $\psi:\R^n\to \R^m$ is said to be subanalytic if $\mathrm{gph}(\psi)$ is subanalytic.
\end{definition}
}
The \revise{proof of the next proposition} is inspired by \cite[Lemma 11]{cheridito2022gradient}.

\begin{prop}
  \label{prop4-3}
     Suppose that $F$ in \eqref{relaxprob1} is a \revise{subanalytic} function. Then, there exists a full measure set $\Theta\subseteq \R_{++}$ on $\R_{++}$, such that if $\theta_0\in \Theta$, then $(\rm{A1})$ holds.
\end{prop}
\begin{proof}
    Define $\cH:\R^n\times \R^n\times \R\to \R$ as $\cH(a,b,\theta):=\det(I_{2n}-\theta \nabla^2F(a,b))=\det(D\cG_\theta(a,b))$. \revise{Now, we identify $\R^n\times \R^n\times \R$ as $\R^{2n+1}$ with the usual Lebesgue measure and Euclidean topology.} Since $F$ is \revise{subanalytic and $F\in C^2(\R^n)$, we know each component of $\nabla^2F$ is subanalytic by \cite[(I.2.1.13)]{shiota1997geometry}, and hence $\cH$ is subanalytic by \cite[(I.2.1.9)]{shiota1997geometry} and the product form of determinant. Therefore, the set $Z:=\cH^{-1}\{0\}$ is closed by the continuity of $\cH$ and also subanalytic by \cite[(I.2.1.4)]{shiota1997geometry}. Applying \cite[Lemma I.2.2]{shiota1997geometry}, we see that $Z$ can be represented as disjoint union of $C^2$ submanifolds\footnote{\revise{We apply this lemma with $A_\nu\equiv \R^{2n+1}$ to obtain a Whitney stratification $\{X_i\}_{i\in J}$ of $Z$; see \cite[Page~4]{shiota1997geometry} for the definition of Whitney stratification.}} $\{X_i\}_{i\in J}$ of $\R^{2n+1}$, where $J$ is an index set, and $\{X_i\}_{i\in J}$ is locally finite at each $x\in Z$ in the sense that there is a neighborhood $U$ of $x$ such that only finitely many sets in $\{U\cap X_i\}_{i\in J}$ are nonempty. Since $\R^{2n+1}$ is second countable, we know $J$ is at most countable. Next, we note that for each $(a,b)\in \R^n\times \R^n$, only finitely many $\theta\in \R$ satisfy that $\cH(a,b,\theta)=0$. This means that $Z$ does not contain any open set in $\R^{2n+1}$, and hence for all $i\in J$ it holds that $\dim(X_i)<2n+1$; see \cite[Theorem 1.21]{lafontaine2015introduction}. According to \cite[Proposition 1.38]{lafontaine2015introduction}, we know each $X_i$ with $i\in J$ has measure zero in $\R^{2n+1}$, and hence $Z=\cup_{i\in J} X_i$ has measure zero in $\R^{2n+1}$.} By \cite[Chapter 2, Corollary 3.3]{stein2009real}, we know that there exists a full measure set $Q\subseteq \R$, such that for each $\theta\in Q$, the set $Z_\theta:=\{(a,b)\in\R^n\times\R^n:\,(a,b,\theta)\in Z\}$ has zero measure and is closed as a slice of the closed set $Z$. Therefore, it suffices to take $\Theta=\left(\cap_{\revise{k=0}}^\infty Q/\kappa^k\right)\cap \R_{++}$.
\end{proof}
\begin{thm}
  \label{thm4-4}
    Suppose that \cref{assum3-1} holds. Then there exists a subset $V\subseteq \R^{2n},$ whose complement has zero Lebesgue measure, such that if $(a^0,b^0)\in V$, then the sequence $\{(a^k,b^k)\}$ generated by \cref{algo1} converges to a second-order stationary point $(a^*,b^*)$ of $F$ in \eqref{relaxprob1}. Moreover, in this case, denote $s^*=(a^*)^2-(b^*)^2$ and assume in addition that $f$ in \eqref{l1prob} satisfies the KL property at $s^*$ with exponent $\alpha\in (0,1)$. Then the following statements hold.
    \begin{enumerate}[\rm (i)]
        \item If $-\nabla h(x^*)\in  \ri(\mu\partial\|s^*\|_1)$, then $\|(a^k,b^k)-(a^*,b^*)\|=O(k^{-\frac{1-\alpha}{2\alpha-1}})$ if $\alpha\in (\frac{1}{2},1)$, and $\|(a^k,b^k)-(a^*,b^*)\|=O(c^k)$ for some $c\in (0,1)$ if $\alpha\in (0,\frac12]$.
        \item If $h$ is convex and the conditions in \cref{thm3-11} are all satisfied, then it holds that $\|(a^k,b^k)-(a^*,b^*)\|=O(k^{-\frac{1-\beta}{2\beta}})$, where $\beta$ is defined in \cref{thm3-11}.
    \end{enumerate}
\end{thm}
\begin{proof}
    Since $F$ is level bounded and $\{F(a^k,b^k)\}$ is nonincreasing, we see that $\{(a^k,b^k)\}$ lies in the bounded set $\{(a,b)\in\R^{n}\times \R^n:~F(a,b)\leq F(a^0,b^0)\}$. Therefore, the stepsize $\theta$ in \cref{algo1} would remain constant eventually; see the proof of \cite[Proposition~A.1(ii)]{chen2016penalty}. The result about the global convergence to a stationary point
    now follows from (A3) of \cref{assum3-1} and standard arguments as in \cite{AttBolSva13}. In addition, the convergence rate result will follow (from, e.g., \cite{attouch2009convergence}) once $(a^*,b^*)$ is shown to be a second-order stationary point so that the KL exponent of $F$ at $(a^*,b^*)$ can be inferred from that of $f$ at $s^*$ according to \cref{subs:uncontrainKLregular} and \cref{subs:uncontrainklgeneral}.
    In other words, now it suffices to argue for the existence of the set $V$.

    To this end, let $A$ be defined as the set of all the strict saddle points of $F$. For each $\theta\in \Gamma$ defined in \cref{assum3-1}, we define the set $B_{\theta}:=\{(a,b)\in \R^{2n}:~\lim_{k\to\infty}\cG_\theta^k((a,b))\in A\}$. Apply \cite[Proposition 7]{cheridito2022gradient},\footnote{\revise{Specifically, we apply \cite[Proposition 7]{cheridito2022gradient} to the mapping $\cG_\theta$. Notice that \cite[Proposition 7]{cheridito2022gradient} requires the function $f:\R^d\to\R^d$ there to have locally Lipschitz Jacobian, but this condition was only used to show the next statement: ``for any fixed $z\in \R^n$, given any $\epsilon>0$, there exists $r_\epsilon>0$ such that $R(x):= f(x)-z-Df(z)(x-z)$ is $\epsilon$-Lipschitz on $\bB(z,r_\epsilon)$". We point out here that merely requiring $f$ there to be continuously differentiable would suffice. Indeed, we can choose $r_\epsilon>0$ such that $\sup_{y^1,y^2\in \bB(z,r_\epsilon)}\|Df(y^1)-Df(y^2)\|_2\leq\epsilon$, where $\|\cdot\|_2$ denotes the spectral norm. Then, taking arbitrary $x^1,x^2\in \bB(z,r_\epsilon)$, by calculus we have $\|R(x^1)-R(x^2)\|=\|f(x^1)-f(x^2)-Df(z)(x^1-x^2)\|=\|\int_0^1 (Df(x^2+t(x^1-x^2))-Df(z))(x^1-x^2)dt\|\leq \int_0^1 \|Df(x^2+t(x^1-x^2))-Df(z)\|_2\|x^1-x^2\|dt\leq \epsilon \|x^1-x^2\|$, which proves that $R$ is $\epsilon$-Lipschitz continuous on $\bB(z,r_\epsilon)$. Since our $\cG_\theta$ is continuously differentiable because $h\in C^2(\R^n)$, \cite[Proposition 7]{cheridito2022gradient} is applicable.}} we know $B_\theta$ has measure zero for each $\theta\in \Gamma$. Therefore, the set $B:=\cup_{\theta\in\Gamma}B_\theta$ also has measure zero. Let $Q$ be the set consisting of all finite sequences of elements in $\Gamma$. According to the argument in \cite[Theorem~2.1-1]{ciarlet2013linear}, we know $Q$ is also countable. Since the stepsize $\theta$ in \cref{algo1} would only change finitely many times, we have:\vspace{-0.1 cm}
    \[
    \vspace{-0.1 cm}
    \textstyle
    \{(a^0,b^0):\lim_{k\to\infty}(a^k,b^k)\in A\}\subseteq \bigcup_{(\theta_1,\dots,\theta_r)\in Q}\cG_{\theta_1}^{-1}\circ\dots\circ \cG_{\theta_r}^{-1}(B),
    \]
    where the right hand side has measure zero as countable union of measure zero sets: sets of the form $\cG_{\theta_1}^{-1}\circ\dots\circ \cG_{\theta_r}^{-1}(B)$ have zero Lebesgue measures thanks to (A1) and \cite[Lemma 6]{cheridito2022gradient}. Therefore, it suffices to take $V=[\bigcup_{(\theta_1,\dots,\theta_r)\in Q}\cG_{\theta_1}^{-1}\circ\dots\circ \cG_{\theta_r}^{-1}(B)]^c$.
\end{proof}
\revise{
\begin{corollary}
  Assume that $h$ in \eqref{l1prob} is subanalytic and lower bounded, and the initial stepsize $\theta_0$ and the initial point $(a^0,b^0)$ in \cref{algo1} are chosen independently according to distributions whose cumulative distribution functions (regarded as probability measures) are absolute continuous with respect to the Lebesgue measures on $\R_{++}$ and $\R^n\times \R^n$, respectively. Then, with probability 1, the sequence $\{(a^k,b^k)\}$ generated by \cref{algo1} converges to a second-order stationary point $(a^*,b^*)$ of $F$ in \eqref{relaxprob1}. Moreover, in this case, denote $s^*=(a^*)^2-(b^*)^2$ and assume in addition that $f$ in \eqref{l1prob} satisfies the KL property at $s^*$ with exponent $\alpha\in (0,1)$. Then the convergence rate of $\{(a^k,b^k)\}$ follows the rate given in \cref{thm4-4}.
\end{corollary}
\begin{proof}
   In view of \cref{prop4-3}, \cref{thm4-4} and \cite{bolte2007lojasiewicz}, it suffices to verify that $F$ is subanalytic, which follows from \cite[(I.2.1.9) and (I.2.1.10)]{shiota1997geometry}.
\end{proof}
}

\appendix

\section{\revise{Proof of \cref{prop2-2}}}
\label{appendixA}
\begin{proof}
  Fix any $x,y\in \cX$. Since $g\in C^2(\R^n)$, we can assume that $\nabla g$ is Lipschitz continuous with a modulus of $L'$ on a bounded open convex neighborhood of the line segment $[x,y]$. We choose $L > L'$ to be sufficiently large so that this neighborhood contains $y-\frac{1}{L}(\nabla g(y)-\nabla g(x))$.

  Let $\phi(z)=g(z)-\langle \nabla g(x),z\rangle$. By the first-order optimality condition, we know that $x\in \argmin_{z\in\R^n}\phi(z)$. In particular, $\phi(x)\leq \phi(y-\frac{1}{L}\nabla \phi(y))$. Applying the descent lemma, we deduce that $\phi(x)\leq \phi(y)-\frac{1}{2L}\|\nabla \phi(y)\|^2$.
    Since $\phi(z)=g(z)-\langle \nabla g(x),z\rangle$, we obtain immediately that
   \begin{equation}\label{lip}
   \textstyle\frac{1}{2L}\|\nabla g(x)-\nabla g(y)\|^2+g(x)+\langle \nabla g(x),y-x \rangle\leq g(y).
   \end{equation}
   Next, since $x\in \cX$, the first-order optimality condition says that $-\nabla g(x)\in \partial \vp(x)$, to which we can apply the convexity of $\vp$ to deduce:
   \begin{equation}\label{lip1}
   \vp(x)+\langle -\nabla g(x), y-x\rangle \leq \vp(y).
   \end{equation}
   Summing \eqref{lip} and \eqref{lip1}, we see further that
   $\psi(x)+  \frac{1}{2L}\|\nabla g(x)-\nabla g(y)\|^2\leq \psi(y)$.
   Notice that $\psi(x)=\psi(y)$ due to $x,y\in \cX$. Hence, we have
   $\frac{1}{2L}\|\nabla g(x)-\nabla g(y)\|^2\leq 0$,
   which implies that $\nabla g(x)=\nabla g(y)$. Consequently, $\nabla g$ is constant on the set $\cX$, and $\cX\subseteq\nabla g^{-1}(v)\cap \partial \vp^{-1}(-v)$. The converse inclusion follows from the first-order sufficient condition for convex optimization problems and the fact that $\partial \psi=\nabla g+\partial\vp$.
 \end{proof}
\section{Proof of \cref{prop3-10}}
\label{appendixB}
\begin{proof}
  From \cite[Exercise~8.8(c)]{rockafellar2009variational}, we have $\partial \psi=\nabla g+\partial \vp$. Hence, the stationarity of $\bar x$ implies that $-\nabla g(\bar x)\in\partial \vp(\bar x)$. Without loss of generality, we assume $\bar x=0$, and replace $g,\vp$ by $g(\cdot)-\langle \nabla g(0),\cdot \rangle$ and $\vp(\cdot)+\langle \nabla g(0),\cdot \rangle$, respectively. Then we have
  \begin{equation}
  \label{tildepsi}
      -\nabla g(0) = 0\in \partial\vp(0),~\tilde \psi=g+\iota_K.
  \end{equation}
 Before proving items (i)-(iii), we first establish two auxiliary facts. Our first goal is to demonstrate the existence of a neighborhood $V$ of $(0,0)$ such that
   \[
   V\cap \gph\partial \vp=\revise{V}\cap \gph\partial \iota_K.
   \]
    According to \cite[equation (4.2.7)]{facchinei2007finite} (see also \cite[Example 5.3.17]{milzarek2016numerical}), since $\vp$ is polyhedral with $0 \in \dom \vp$, there exists a neighborhood $V_1$ of $0$ such that

    \[
    \revise{\vp(x)=p(x)+\vp(0)\quad \forall x\in V_1, \ \ {\rm where}\ \ p(x) := \vp'(0;x).}
    \]
    This equation also establishes that
    \begin{equation}
    \label{plsc}
    \text{$p$ is proper, closed and convex, and }\partial\vp(x)=\partial p(x)~~\forall x\in V_1.
    \end{equation}
    Hence, by \cite[Theorem 11.1]{rockafellar2009variational} we have $p^{**}=p$, where $p^*$ is the convex conjugate of $p$. Applying \cite[Proposition~17.17]{bauschke2017convex}, it follows that
    \begin{equation}
        \label{psupp}
            p=p^{**}=(\iota_{\partial \vp(0)})^{*}=\sigma_{\partial \vp(0)}.
    \end{equation}
    Let $C=\partial \vp(0)$, we know that $C$ is a polyhedral set, as stated in \cite[Proposition 10.21]{rockafellar2009variational}. Moreover, we have $-\nabla g(0)=0\in \partial \vp(0)=\partial p(0)=C$. By utilizing \cite[Exercise 6.47]{rockafellar2009variational}, we know that there  exists an open neighborhood $V_2$ of $0$ such that
    \[
    T_{C}(0)\cap V_2=C\cap V_2.
    \]
    Consequently, for each $y\in V_2$, we have $\partial \iota_{T_{C}(0)}(y)=\partial \iota_{C}(y)$, as the subgradient can be defined locally. This implies that
    \[
    (V_2\times \R^n)\cap \gph \partial \iota_{T_{C}(0)}=(V_2\times \R^n)\cap \gph\partial\iota_{C}.
     \]
     From \cite[Theorem 23.5]{rockafellar1970convex}, we see further that
     \[
     (\R^n \times V_2 )\cap \gph\partial \sigma_{T_{C}(0)}=(\R^n\times V_2)\cap \gph\partial\sigma_{C}.
     \]
     Notice that $\sigma_{T_C(0)}=\iota_{N_C(0)}=\iota_K$, by \eqref{plsc} and \eqref{psupp} we conclude that
     \begin{equation}
         \label{neighborsubgrad}
      V\cap \gph\partial \iota_{K}=V\cap \gph\partial \vp,
     \end{equation}
     where $V=V_1\times V_2$.

    The second goal is to show that,
    \begin{equation}
        \label{sec_goal}
        \vp(x)-\vp(0)=\iota_K(x)-\iota_K(0)\ \ \ \forall (x,v)\in V\cap \gph\partial \vp=V\cap \gph\partial \iota_K.
    \end{equation}
    Take $(x,v)\in V\cap \gph\partial \vp=V\cap \gph\partial \iota_K$, we know $x\in\dom\partial\iota_K=\dom\iota_K$. Since $\partial\iota_K = N_K$ and $V=V_1\times V_2$ with $V_2$ being a neighborhood of $0$, this implies that\vspace{-0.1 cm}
    \[\vspace{-0.1 cm}
    (x,0)\in V\cap \gph\partial \iota_K = V\cap \gph\partial \vp,
    \]
    where we have used \eqref{neighborsubgrad} for the equality. The above display implies $0 \in \partial \vp(x)$. Using this together with $0 \in \partial\vp(0)$ (see \eqref{tildepsi}) and the convexity of $\vp$, we see that $\vp(x)=\vp(0)$, which means that $\vp(x)-\vp(0)=0=\iota_K(x)-\iota_K(0)$.

    We are now ready to prove (i)--(iii). We start with (i). Since $\nabla g$ is continuous and $\nabla g(0)=0$ (see \eqref{tildepsi}), we can choose a sufficiently small $a > 0$ such that
    \begin{equation}
        \label{eq_V}
        (\bB_{a},-\nabla g(\bB_a)+\bB_a)\subseteq V.
    \end{equation}
    Take $U=\bB_a\times \bB_a$. Suppose $(x,v)\in U\cap \gph\partial\psi$, then it holds that
    \[
    (x,v-\nabla g(x))\overset{\rm(a)}{\in} \gph\partial \vp,\quad (x,v-\nabla g(x))\overset{\rm(b)}{\in} (\bB_{a},-\nabla g(\bB_a)+\bB_a) \overset{\rm(c)}{\subseteq} V,
    \]
    where (a) holds because $\partial \psi=\nabla g+\partial \vp$, (b) holds as $x$, $v\in \bB_a$ and (c) is due to \eqref{eq_V}. Consequently, we know that
    \begin{equation}\label{finallyadded}
    (x,v-\nabla g(x))\in V\cap \gph\partial\vp =V\cap \gph\partial\iota_K,
    \end{equation}
    where we have used \eqref{neighborsubgrad}. Then, we know $(x,v)\in U\cap \gph\partial\tilde\psi$ since $\partial\tilde\psi=\nabla g+\partial\iota_K$. This proves $U\cap \gph\partial\psi\subseteq U\cap\gph\partial\tilde\psi$. A similar argument shows that $U\cap \gph\partial\psi\supseteq U\cap\gph\partial\tilde\psi$, which proves that $U\cap \gph\partial\psi=U\cap\gph\partial\tilde\psi$.

    Next, notice that for all $(x,v)\in U\cap \gph\partial\psi= U\cap\gph\partial\tilde\psi$, we can deduce \eqref{finallyadded} as argued above. This together with \eqref{sec_goal} shows that
    \[
    \psi(x)-\psi(0)-\left(\tilde \psi(x)-\tilde \psi(0)\right)= \vp(x)-\vp(0)-\left(\iota_K(x)-\iota_K(0)\right)=0.
    \]
    This proves (i). Part (ii) follows from (i) and \cref{rem2-3}.

    Finally, for part (iii), assume that $0$ is a local minimizer of $\psi$, then we may take a sufficiently small neighborhood $U_1$ of $0$ such that for all $x\in U_1$, we have $\psi(x)\geq \psi(0)$, $(x,0)\in U$ and $(x,\nabla g(x))\in U$, where $U$ is given in item (i). Then, for all $x\in U_1\cap \dom \tilde\psi $, notice that $\nabla g(x)+0\in \nabla g(x)+\partial\iota_K(x)=\partial \tilde \psi(x)$, we have
    \[
    (x,\nabla g(x))\in U\cap \gph\partial\tilde\psi\overset{\rm (a)}=U\cap \gph\partial \psi,
    \mbox{ and hence }\tilde\psi(x)-\tilde\psi(0)\overset{\rm (b)}=\psi(x)-\psi(0)\ge 0,
    \]
    where we have used item (i) for (a) and (b). This means that $0$ is a local minimizer of $\tilde\psi$. Moreover, if $x\in U_1\cap \Omega$, then by the first-order optimality condition, we see that $(x,0)\in U\cap \gph\partial\psi=U\cap\gph\partial\tilde\psi$, and hence $\tilde \psi(x)-\tilde\psi(0)=\psi(x)-\psi(0)=0$ by item (i). This shows that $U_1\cap\Omega\subseteq\tilde\Omega$. A similar argument shows that $U_1\cap\tilde\Omega\subseteq U_1\cap \Omega$, and proves that $\Omega$ locally agrees with $\tilde\Omega$.
\end{proof}

\section{Proof of \cref{lemma2-5}}
\label{appendixC}
\begin{proof}
  \revise{
  There is no loss of generality if we assume $\bar x=0$. Then $g(0)=\psi(0) = \inf \psi$. In view of the KL assumption and \cref{prop3-12}, by shrinking the $U$ in \eqref{J3eb} further if necessary, we can deduce that there exists $c_1 > 0$ such that
  \begin{equation}\label{eqeb}
    g(x) - g(0)\geq c_1^{-\frac1{1-\alpha}}\dist^{\frac{1}{1-\alpha}}(x,\Omega)\ \ \ \forall x\in U\cap K.
  \end{equation}
  Shrink $U$ further if necessary such that
  \begin{equation}
      \label{condU}
     g(x)-g(0)\leq 1\ \ \forall x\in U\cap K.
  \end{equation}

  Fix any $x\in U\cap K$ with $g(x)>g(0)$. We define
  \begin{equation}\label{eqindex}
  I:=\left\{i\in J_3:~|x_i|\leq c_2(g(x) - g(0))^{\gamma(1-\alpha)}\right\},
  \end{equation}
  where $c_2=cc_1^\gamma$, with $c_1$ defined in \eqref{eqeb} and $c$ and $\gamma$ defined in \eqref{J3eb}. We define $\rho\in \R_{+}^{J_3}$ as $\rho_i=|x_i|$ for all $i\in J_3$. Then, by \cref{J3eb} and notice that $x\in S_\rho$, we see that
  \begin{equation}\label{xomegaieb}
    \begin{aligned}
    \dist(x,\Omega\cap S_{\rho}) &\leq c\max\{\dist(x,S_{\rho}),\dist(x,\Omega)\}^\gamma=c\,\dist^\gamma(x,\Omega)\\
    &\overset{\rm (a)}{\leq} cc_1^\gamma(g(x)-g(0) )^{\gamma({1-\alpha})}=c_2 (g(x)-g(0) )^{\gamma(1-\alpha)},
    \end{aligned}
\end{equation}
where (a) follows from \eqref{eqeb}. Additionally, let $\hat x=P_{\Omega\cap S_{\rho}}(x)$, and we know that $\|x-\hat x\|=\dist(x,\Omega\cap S_{\rho})$. We further define the following three index sets:
       \begin{equation}
        \label{defn_indexi3}
        I_1:=\{i\in I:~\rho_i=0\},~I_2:=I\setminus I_1,~I_3:=J_3\setminus I.
       \end{equation}
      Then for $i\in I_2\cup I_3$ we have $|x_i|>0$ due to $\rho_i=|x_i|$ for $i\in J_3$ and \cref{eqindex}, and moreover,
       \begin{equation}
        \label{boundxi}
        \begin{aligned}
          &\forall i\in I_1,~ \hat x_i=\rho_i=0,~|x_i-\hat x_i|=|x_i|\overset{\rm{(a)}}{\leq} c_2(g(x) - g(0))^{\gamma(1-\alpha)},\\
          &\forall i\in I_2,~ \frac{1}{|x_i|} |x_i-\hat x_i|^2\overset{\rm{(b)}}{\leq} \frac{4}{|x_i|} |x_i|^2=4|x_i|\overset{\rm{(c)}}{\leq} 4c_2(g(x) - g(0))^{\gamma(1-\alpha)},  \\
          &\forall i\in I_3,~\frac{1}{|x_i|} |x_i-\hat x_i|^2\overset{\rm (d)}{\leq} \frac{c_2^2(g(x)-g(0))^{2\gamma(1-\alpha)}}{|x_i|}\overset{\rm (e)}{\leq} c_2(g(x) - g(0))^{\gamma(1-\alpha)},
        \end{aligned}
       \end{equation}
       where in (a) and (c) we have used the fact that $I_1,I_2\subseteq I$ in \eqref{defn_indexi3} and the definition of $I$ in \eqref{eqindex}, (b) follows from the facts that $\hat x\in S_{\rho}$ and that for all $y\in S_{\rho},~i\in J_3$ it holds that $|x_i-y_i|\leq |x_i|+|y_i|\leq 2|x_i|$, (d) holds because of the relation $|x_i-\hat x_i|\leq \|x-\hat x\|=\dist(x,\Omega\cap S_\rho)$ and the bound in \eqref{xomegaieb}, and (e) follows from the definition of $I_3$ in \eqref{defn_indexi3} and that of $I$ in \eqref{eqindex}. Therefore, by using the convexity of $g$, we can deduce that:
           \begin{align*}
            &g(x)-g(0) =\psi(x)-\psi(0) = \psi(x) - \psi(\hat x) = g(x) - g(\hat x) \\
             &\overset{\rm(a)}{\leq}  \sum_{i=1}^n|\nabla g(x)_i||x_i-\hat x_i|\overset{\rm(b)}{=}\sum_{i\in J_2\cup J_3}|\nabla g(x)_i||x_i-\hat x_i| \\
             &\overset{\rm(c)}{=} \sum_{i\in J_2}|\nabla g(x)_i||x_i-\hat x_i|+\sum_{i\in I_1} |\nabla g(x_i)||x_i|^{\frac12}|x_i|^{\frac12}\!\! +\!\!\!\!\sum_{i\in I_2\cup I_3}|\nabla g(x)_i|\sqrt{|x_i|} \frac{1}{\sqrt{|x_i|}}|x_i-\hat x_i| \\
             &\overset{\rm(d)}{\leq} \sqrt{\sum_{i\in J_2}|\nabla g(x_i)|^2+\sum_{j\in J_3}|x_i||\nabla g(x_i)|^2 }\sqrt{\sum_{i\in J_2}|x_i-\hat x_i|^2+\sum_{i\in I_1}|x_i|+\sum_{i\in I_2\cup I_3}\frac{|x_i-\hat x_i|^2}{|x_i|}}  \\
             &\overset{\rm(e)}{\leq} c_4\sqrt{\sum_{i\in J_2}|\nabla g(x_i)|^2+\sum_{j\in J_3}|x_i||\nabla g(x_i)|^2 }\cdot (g(x)-g(0))^{\frac{\gamma(1-\alpha)}{2}}
            \end{align*}
  where (a) follows from the convexity of $g$, which implies that $g(x)-g(\hat x)\leq \langle \nabla g(x), x-\hat x\rangle$, (b) follows from condition (ii), which implies $x_i=\hat x_i$ for all $i\in J_1$,\footnote{\revise{Recall that we assumed $\bar x = 0$. Then both $x$ and $\hat x \in 0 + K$, which means $x_i = \hat x_i$ for $i \in J_1$ according to condition (ii).}} (c) follows from the definition of the index sets $I_1,I_2$, and $I_3$ in \eqref{defn_indexi3}, (d) follows from the Cauchy-Schwarz inequality,  (e) follows from the facts that $\sum_{i\in J_2}|x_i-\hat x_i|^2\leq \|x-\hat x\|^2\leq c_2^2(g(x)-g(0))^{2\gamma(1-\alpha)}\leq c_2^2(g(x)-g(0))^{\gamma(1-\alpha)}$ (which follows from \eqref{xomegaieb} and \eqref{condU}), \eqref{boundxi}, and we set $c_4=2\sqrt{n}\max\{c_2,\sqrt{c_2}\}$. Rearrange this inequality and notice that $\beta=1-\gamma(1-\alpha)$, we obtain
  \begin{align*}
    \frac{1}{c_4^2}(g(x)-g(0))^{1+\beta} \leq   \sum_{i\in J_2}|\nabla g(x_i)|^2+\sum_{j\in J_3}|x_i||\nabla g(x_i)|^2 .
  \end{align*}
}
\end{proof}
\bibliographystyle{siamplain}
\bibliography{Commonbib}
\end{document}